\documentclass[12pt]{article}
\usepackage{amsfonts,amsmath,latexsym,amssymb} 
\usepackage{graphicx}
\usepackage{wrapfig}
\usepackage{color}

\newenvironment{proof}{\noindent{{\bf Proof}:\ }}{\vspace{2ex}}

\def\numberlikeadb{\global\def\theequation{\thesection.\arabic{equation}}}
\numberlikeadb
\newtheorem{theorem}{Theorem}[section]
\newtheorem{lemma}[theorem]{Lemma}
\newtheorem{corollary}[theorem]{Corollary}
\newtheorem{proposition}[theorem]{Proposition}

\newcommand{\var}{{\mathrm{Var}}}
\newcommand{\Cov}{{\mathrm{Cov}}}

\newcommand{\beas}{\begin{eqnarray*}}

\newcommand{\enas}{\end{eqnarray*}}
\newcommand{\eqs}{\begin{eqnarray*}}
\newcommand{\ens}{\end{eqnarray*}}

\newcommand{\bea}{\begin{eqnarray}}
\newcommand{\ben}{\end{eqnarray}}

\newcommand{\eqa}{\begin{eqnarray}}

\newcommand{\ena}{\end{eqnarray}}
\newcommand{\eq}{\begin{equation}}
\newcommand{\en}{\end{equation}}

\newcommand{\halmos}{\rule{1ex}{1.4ex}}

\def\half{{\textstyle{\frac12}}}

\def\third{{\textstyle{\frac13}}}

\def\quarter{{\textstyle{\frac14}}}
\def\eighth{{\textstyle{\frac18}}}
\def\Ref#1{(\ref{#1})}
\def\a{\alpha}
\def\b{\beta}
\def\s{\sigma}

\def\l{\lambda}

\def\p{\pi}

\def\sjn{\sum_{j=1}^n}

\def\dtv{d_{\mathrm{TV}}}
\def\law{{\cal L}}

\def\ep{\hfill $\proofbox$ \bigskip}

\def\re{{\mathbb R}}
\def\giv{\,|\,}

\def\non{\nonumber}

\def\e{\varepsilon}
\def\m{\mu}
\def\var{{\rm Var\,}}
\def\r{\rho}

\def\t{\tau}
\def\g{\gamma}
\def\h{\eta}

\def\ps{\psi}

\def\Bl{\left(}
\def\Br{\right)}

\def\Blb{\left\{}
\def\Brb{\right\}}

\def\Bi{{\rm Bi\,}}

\def\nin{\noindent}
\def\D{\Delta}
\def\ex{{\mathbb E}}
\def\pr{{\mathbb P}}

\def\Be{{\rm Be\,}}
\def\ff{{\cal F}}
\def\d{\delta}

\def\Eq{\ =\ }
\def\Le{\ \le\ }

\def\G{\Gamma}

\def\ui{^{(1)}}

\def\bgamma{{\mathbf z}}

\def\uj{^{(j)}}

\def\ul{^{(l)}}
\def\uii{^{(i)}}

\def\ignore#1{}

\def\cov{\Cov}
\def\ch{\chi}

\def\bone{{\bf{1}}}

\def\integ{{\mathbb Z}}

\def\nti{n\to\infty}

\def\Def{\ :=\ }

\def\ut{^{(2)}}
\def\uid{^{(i')}}

\def\tW{{\widetilde W}}

\def\hZ{{\widehat Z}}

\def\l{\lambda}

\def\Z{\integ}

\def\ur{^{(r)}}
\def\S{\Sigma}

\def\bpi{{\boldsymbol{\pi}}}

\def\ccc{ C}

\def\tm{{\tilde \m}}
\def\tmm{{\tilde m}}
\def\tv{{\tilde v}}

\def\ujk{^{(j,k)}}

\def\andd{\quad \mbox{and}\quad }

\def\leqn{\lefteqn}

\def\tY{{\widetilde Y}}
\def\tG{{\widetilde G}}

\def\tX{{\widetilde X}}

\def\cc{c}
\def\DN{{\rm DN}}
\def\tr{{\rm Tr}}
\def\hn{{m}}
\def\ABA{{\widetilde\AA}}
\def\eii{e\uii}
\def\AA{{\cal A}}
\def\tih{{\tilde h}}
\def\BLX{Barbour, Luczak \& Xia}
\def\NN{{\cal N}}
\def\ej{e^{(j)}}
\def\Rh{\r}

\def\ep{\hfill$\halmos$}
\def\tbd{{\widetilde D}}
\def\tbdt{{\widetilde{D^2}}}

\def\bdt{{\overline{D^2}}}
\def\ur{^{(r)}}
\def\gg{{\cal G}}
\def\ud{^{(d)}}
\def\hR{{\widehat R}}
\def\hmin{h_{\rm min}}
\def\tV{{\widetilde V}}
\def\te{{\tilde\e}}
\def\hpr{{\hat p}_r}
\def\QQ{{\cal Q}}
\def\sshn{{\half\hn\dnew}}
\def\shn{{\sqrt\hn}}
\def\hV{{\widehat V}}
\def\dnew{\d}
\def\uk{^{(k)}}

\def\caB{{\cal B}}
\def\caF{{\cal F}}

\def\caH{{\cal H}}
\def\esup{{{\rm esssup}}}
\def\uja{^{(\a)}}
\def\ujkab{^{(\a,\b)}}
\def\ujkaa{^{(\a,\a)}}
\def\ukb{^{(\b)}}
\def\intG{{\int_\G}}
\def\Ups{{\Upsilon}}
\def\ups{{\upsilon}}
\def\tXi{{\tilde\Xi}}
\def\qm{q}

\newenvironment{prooftxia}{{\bf Proof of Theorems~\ref{DN-approx}
 and \ref{DN-approx-cont}}}{$\Box$ \vspace{1ex}}

\newcommand{\qed}{\nopagebreak\hspace*{\fill}
{\vrule width6pt height6ptdepth0pt}\par}

\makeatletter
\renewcommand\theequation{\thesection.\@arabic\c@equation}
\makeatother

\begin{document}

\title{Multivariate approximation in total variation using local dependence}
\author{A. D. Barbour\footnote{Institut f\"ur Mathematik, Universit\"at Z\"urich,
Winterthurerstrasse~190, 8057 Z\"urich, Switzerland;
e-mail {\tt a.d.barbour@math.unizh.ch}.
Work carried out in part at the University of Melbourne and at Monash University, and
supported in part by Australian Research Council Grants Nos DP150101459 and DP150103588.}
\andd   
A. Xia\footnote{School of Mathematics and Statistics, University of Melbourne, Parkville, VIC 3010, Australia;
e-mail {\tt aihuaxia@unimelb.edu.au}.
Work supported in part by Australian Research Council Grant No.\ DP150101459.}  
\\[1ex]   
Universit\"at Z\"urich  and  University of Melbourne}    
\date{\today}
\maketitle

\begin{abstract} 
We establish two theorems for assessing the accuracy in total variation of multivariate discrete 
normal approximation 
to the distribution of an integer valued random vector $W$. The first is for sums of random vectors whose
dependence structure is local.  The second applies to random vectors~$W$ resulting from integrating the $\Z^d$-valued 
marks of a marked point process with respect to its ground process. The error bounds are of magnitude comparable
to those given in Rinott \& Rotar~(1996), but now with respect to the stronger total variation distance.
Instead of requiring the summands to be bounded, we make third moment assumptions.
We demonstrate the use of the theorems in four applications: monochrome edges in vertex coloured graphs, 
induced triangles and $2$-stars in random geometric graphs, the times spent in different states by an 
irreducible and aperiodic finite Markov chain, and
the maximal points in different regions of a homogeneous Poisson point process.
\end{abstract} 

\vskip12pt \noindent\textit {Key words and phrases\/}: 
Total variation approximation,  Stein's method, local dependence, marked point process.

\vskip12pt \noindent\textit{AMS 2010 Subject Classification\/}:
Primary 60F05;
secondary 60E15, 60G55, 60J27.

\section{Introduction}
\setcounter{equation}{0}
In this paper, we prove a general theorem that can be used to give bounds in total variation on
the accuracy of multivariate discrete normal approximation to the distribution
of a random vector $W$ in~$\Z^d$, when~$W$ is a sum of~$n$ random vectors whose
dependence structure is local. Our setting is rather similar to that in Rinott \& Rotar~(1996).
In their paper, Stein's method is used to derive the accuracy, in terms of the convex sets metric, 
of multivariate normal approximation to suitably normalized sums of {\it bounded\/} random vectors; under
reasonable conditions, error bounds of order $O(n^{-1/2}\log n)$ are obtained.
Fang~(2014) improves the order of the error to $O(n^{-1/2})$, using slightly different conditions,
and also obtains optimal dependence on the dimension~$d$.
Here, we are interested in total variation distance bounds, so as to be able to approximate the probabilities 
of {\it arbitrary\/} sets.  For random elements of~$\Z^d$, this necessitates replacing the multivariate
normal distribution by a discretized version.  We use the $d$-dimensional discrete normal 
distribution $\DN_d(nc,n\S)$ that is obtained from 
the multivariate normal distribution $\NN_d(nc,n\S)$ by assigning the probability of 
the $d$-box 
\[
    [i_1 - 1/2,i_1 + 1/2) \times\cdots \times [i_d-1/2,i_d+1/2)
\]
to the integer vector $(i_1,\ldots,i_d)^T$, for each $(i_1,\ldots,i_d)^T\in \Z^d$. 
This family of distributions is a natural choice, when approximating a
discrete random vector in a central limit setting.  We are able to establish discrete normal
approximation under conditions broadly analogous to those of Rinott \& Rotar~(1996) and Fang~(2014),
with an error of order $O(n^{-1/2}\log n)$, but without their boundedness assumption;  a suitable
third moment condition is all that is needed.

For generality, we replace $n$ with an $\hn$ which is essentially the dimension adjusted trace of
the covariance matrix of $W$. Our approach to establishing approximation in total variation by $\DN_d({\hn}c,{\hn}\S)$ is by way of
Stein's method.  Letting $e\uii$ denote the coordinate vector in the $i$-direction, we start with a Stein operator $\ABA_{{\hn}}$ defined by
\eq\label{A-def}
    (\ABA_{{\hn}} h)(z) \Def {\hn}\tr(\S \D^2 h(z)) - (z-{\hn}c)^T\D h(z),\quad z\in\Z^d,
\en
where
\[
     \D_j h(z) \Def h(z+\ej) - h(z);\quad \D^2_{jk}h(z) \Def \D_j(\D_k h)(z).
\]
For any function $h\colon \Z^d \to \re$,  $z\in\Z^d$ and $0 < r \le \infty$, define
\eqa
    |\D h(z)| &:=& \max_{1\le i\le d}|\D_i h(z)|; \quad |\D^2 h(z)| \Def \max_{1\le i,k\le d}|\D^2_{ik} h(z)|; \non \\
   \|\D h\|_{r,\infty} &:=& \max_{z \in \Z^d \cap B_r({\hn}c)}|\D h(z)|; 
            \quad \|\D^2 h\|_{r,\infty} \Def \max_{z \in \Z^d \cap B_r({\hn}c)}|\D^2 h(z)|, \label{h-norms}
\ena 
where $B_r(x) := \{y\in\re^d\colon |y-x| \le r\}$; note that the centre~${\hn}c$ is suppressed
in the norm notation.
Using the operator $\ABA_{{\hn}}$, the following abstract result can be deduced from \BLX~(2018b, 
Theorem~2.4 and 2018a, Remark~4.2).

\begin{theorem}\label{ADB-DN-approx-thm}
Let $W$ be a 
random vector in~$\Z^d$ with mean $\m := \ex W$ and
positive definite covariance matrix $V := \ex\{(W-\m)(W-\m)^T\}$; define 
${\hn} := \lceil d^{-1}\tr V \rceil$, $c := {\hn}^{-1}\m$ and $\S := {\hn}^{-1}V$.
Set $\d_0 := \frac1{72}\,\Rh(\S)^{-3/2}$. 
Then, for any $0 < \d \le \d_0$, there exist 
$C_{\ref{ADB-DN-approx-thm}}(\d), n_{\ref{ADB-DN-approx-thm}}(\d) < \infty$, 
depending continuously on $\d$ and the condition number~$\r(\S)$ of~$\S$, but not on $d$ or~${\hn}$,
with the following property:
if, for some  $\e_1$, $\e_{20}$, $\e_{21}$ and $\e_{22}$, 
and  for some ${\hn} \ge n_{\ref{ADB-DN-approx-thm}}({\d})$,
{\rm 
\begin{enumerate}
 \item $\dtv(\law(W),\law(W+\ej)) \Le \e_1$, for each $1\le j\le d$;
 \item $ |\ex\{\ABA_{{\hn}} h(W)\}I[|W-\m| \le {\hn}\d]| \\[1ex]
     \mbox{}\qquad \Le \e_{20}\|h\|_{3{\hn}\d_0/2,\infty} + \e_{21}{\hn}^{1/2}\|\D h\|_{3{\hn}\d_0/2,\infty} 
       + \e_{22}{\hn}\|\D^2 h\|_{3{\hn}\d_0/2,\infty}$,
\end{enumerate}
}
\nin for all $h\colon \Z^d\to\re$, then it follows that 
\eqs
    \leqn{\dtv(\law(W),\DN_{d}({\hn}c,{\hn}\S))} \\
    &&\quad \Le C_{\ref{ADB-DN-approx-thm}}(\d)(d^4({\hn}^{-1/2}+ \e_1) 
       + \e_{20} + \e_{21} + \e_{22})\log {\hn} .
\ens
\end{theorem}

\nin The unspecified constants can in principle be deduced from the more detailed
information in \BLX~(2018a,b).

Applying the theorem in practice may not be easy.  Condition~(b) is much like the
sort of condition that has to be checked to prove multivariate normal approximation
using Stein's method (Chen, Goldstein \& Shao~(2011, p. 337)), 
with differences and derivatives exchanged, except for the indicator $I[|W-{\hn}c| \le {\hn}\d]|$,
which truncates~$W$ to the ball $B_{{\hn}\d}({\hn}c)$.
The truncation has both good and bad consequences.  It introduces an awkward discontinuity 
inside the expectation, which needs careful treatment in the arguments that follow.
On the other hand, it ensures that all the expectations to be considered are finite,
and that the function~$h$ only has to be evaluated within certain closed balls around
${\hn}c$;  this latter feature is important, because the solutions to the Stein equation
for this problem may grow large as the distance from~${\hn}c$ increases.
Condition~(a) imposes a certain smoothness on the distribution of~$W$.

In Section~\ref{main}, we prove a multivariate approximation theorem, 
Theorem~\ref{DN-approx}, with error
bounds in the total variation distance, that is much simpler to use than Theorem~\ref{ADB-DN-approx-thm}.
The setting is one of predominately local dependence.  The basic elements making up the
error bounds are sums of third moments, similar to those that would be expected to quantify
the error in the CLT for dissociated summands, together with dependence coefficients analogous
to those in Rinott \& Rotar~(1996). 
However, there is an extra quantity~$\e_{W}$ appearing in the bound, which quantifies the smoothness of 
the distribution of~$W$, and which is not as simple to express in concrete terms. We also consider a more 
general setting, in which 
$W$ arises from integrating the marks of a marked point process 
with respect to its ground process on a suitable metric space.  {For integrals of functionals
of a Poisson process, Schulte \& Yukich~(2018a,b) have recently established an order $O(n^{-1/2})$ rate of multivariate
approximation with respect to the convex sets metric, using the Malliavin--Stein approach and second order
Poincar\'e inequalities.  They require somewhat stronger moment assumptions than ours, but, as in the theorems of
Rinott \& Rotar~(1996) and of Fang~(2014), there is no need to bound an analogue of~$\e(W)$.}

In Section~\ref{independence},  we introduce a stronger notion of local dependence, that is convenient for
many applications.  It enables us to give rather simple error bounds, 
in Corollary~\ref{cor1},  expressed in terms of an upper
bound for the maximum of the third moments of the~{$|X\uja|$} and the sizes of the neighbourhoods in the
dependency graph, both being quantities that typically appear in error bounds in the CLT.
It also enables us to give a general result, Theorem~\ref{smoothness}, that is helpful for bounding~$\e_{W}$.  
The effectiveness of our bounds is illustrated in a number
of examples in Section~\ref{examples}.  These also give some insight into why,
in addition to the sort of moment conditions that suffice for approximation in metrics weaker
than total variation,  some smoothness condition is needed.

\section{Main {theorems}}\label{main}
\setcounter{equation}{0}

For the ease of use, we present our main results for the accuracy of multivariate discrete normal approximation 
in two distinct but related settings.
\ignore{ to the distribution of $W$ in two setups,
one is for $W$ as a sum of $n$ random vectors whose dependence structure is local,
and the other is for $W$ as a result from integrating the marks of a marked point process  
in terms of its ground process on a suitable metric space. }
We postpone the proofs of the main theorems to Section~\ref{proofmainresults}.

In the first setting, we suppose that~$W = \sjn X\uj$ is a sum of~$n$  vectors in~$\re^d$.
We assume that there are decompositions of the following form:
\begin{enumerate}
  \item
  For each $1\le j\le n$, we can write $W = W\uj + Z\uj$, where~$W\uj \in \Z^d$ is only weakly dependent on~$X\uj$; 
  \item
  For each $1\le j\le n$, we can write $Z\uj = \sum_{k=1}^{n_j} \tX\ujk$, with $\tX\ujk \in \Z^d$,
   and then, for each $1\le k\le n_j$,
   we can write $W\uj = W\ujk + Z\ujk$, where~$W\ujk \in \Z^d$ is only weakly dependent on~$(X\uj,{\tX}\ujk)$.
\end{enumerate}
Because of the restrictions to~$\Z^d$, centring on the mean is not possible in these decompositions,
but it could, for {instance}, be arranged that each component of $Z\uj$, $\tX\ujk$ and~$Z\ujk$ has mean with modulus
at most~$1$.  This makes no difference to the arguments that follow, but 
the moment sums $H_1$ and~$H_2$ that appear in the error bounds might otherwise be larger than necessary.

Weak dependence is expressed by the smallness of dependence coefficients analogous to those in
Rinott \& Rotar~(1996).  With $\m\uj := \ex X\uj$, we begin by defining
{\eqa
    \ch_{12j}  &:=&  \ex\bigl|\ex(|X\uj| \giv W\uj) - \ex |X\uj| \bigr|;
              \non\\
   \ch_{13j}  &:=&  \ex\bigl|\ex(|X\uj|_1 \giv W\uj) - \ex |X\uj|_1 \bigr|;
              \label{Chi-j}\\ 
   \ch_{2jk} &:=&  \sum_{i=1}^d \sum_{l=1}^d \ex\bigl|\ex\{|X_i\uj|\,|\tX_l\ujk| \giv W\ujk\}
                    - \ex\{|X_i\uj|\,|\tX_l\ujk|\}\bigr|\non\\
                    &&\ \ \ + \sum_{i=1}^d \sum_{l=1}^d |\m_i\uj|\ex\bigl|\ex\{|\tX_l\ujk| \giv W\ujk\}
                    - \ex\{|\tX_l\ujk|\}\bigr|\non\\
                    && \ \ \ +\sum_{i=1}^d \sum_{l=1}^d \ex\bigl|\ex\{X_i\uj \,\tX_l\ujk \giv W\ujk\}
                    - \ex\{X_i\uj\,\tX_l\ujk\}\bigr|\\
                    && \ \ \ +\sum_{i=1}^d \sum_{l=1}^d |\m_i\uj|\ex\bigl|\ex\{\tX_l\ujk \giv W\ujk\}
                    - \ex\{\tX_l\ujk\}\bigr|,
    \non
\ena}
and then set
\eqa
   \ch_{11} &:=& (d\hn)^{-1/2} \sjn \ex|\ex(X\uj \giv W\uj) - \ex X\uj|; \non\\
   \ch_{12}  &:=& (d\hn)^{-1/2} \sjn \ch_{12j}; \qquad\qquad \ch_{13}  \Eq d^{-1}\hn^{-1/2} \sjn \ch_{13j};
              \non\\          
  \ch_{2} &:=& d^{-3}\hn^{-1}\sjn \sum_{k=1}^{n_j} \ch_{2jk};  \label{RR-csts}\\
  \ch_3 &:=& d^{-1}\hn^{-1}\sjn \ex\{|\ex(X\uj\giv W\uj)-\m\uj|\,|W\uj-\m|\}.\non
\ena
We then write $\ch_1 := \max_{1\le l\le 3}\ch_{1l}$.  
Note that the $\hn$-factors {defined in Theorem~\ref{ADB-DN-approx-thm}} are not present in the 
quantities in Rinott \& Rotar~(1996) that are
directly analogous to $\ch_{11}$, $\ch_2$ and~$\ch_3$.  This is
because, in their formulation, the random variables corresponding to~$X\uj$ are normalized to make
$\cov(W)$ close to the identity matrix.  Since our sum~$W$ is not normalized, to keep its
values in~$\Z^d$, the elements of its covariance matrix typically grow with~$n$. 
The quantities $\ch_{12}$ and~$\ch_{13}$ have no direct analogue in Rinott \& Rotar~(1996), and
appear only in dealing with the truncation to $B_{n\d}(\m)$, something that is not needed
in their arguments.

Assuming that $\ex|X\uj|^3 < \infty$ for each $j \in [n]{:=\{1,2,\dots,n\}}$, we define
\eqa
   \m &:=& \ex W \Eq \sjn \m\uj; \quad
   V \Def \cov(W), \label{moments}
\ena
and set
\eqa
   \hn &:=& \lceil d^{-1}\tr V \rceil; \quad \cc \Def \hn^{-1}\m;\quad \S \Def \hn^{-1}V, \label{hn-def}
\ena
so that $\tr(\S) = \hn^{-1}\tr V  \le d$; this makes~$\hn$ the analogue of the variance in
the one dimensional context.  
We then introduce some moment sums, used in the error estimates,  
\ignore{
These involve 
random vectors $Z\uj$, $\tX\ujk$ and~$Z\ujk$ that would be centred in many analogous
theorems for normal approximation in~$\re^d$, but need to be integer valued here.
They can be expected to give useful bounds provided, for instance, that the~$X\uj$ are `centred'
so that $|\m_l\uj| \le 1$ , for each $1\le l\le d$ and $1\le j\le n$.
}
defining
\eqa
   H_{21} &:=& d^{-3/2}\hn^{-1}\sjn \ex\{{(|X\uj|+|\m\uj|)}\,|Z\uj|^2\};               \non \\
   H_{22} &:=& d^{-3/2}\hn^{-1}\sjn \sum_{k=1}^{n_j} \ex\{{(|X\uj|+|\m\uj|)}\,|\tX\ujk|\,|Z\ujk|\};   \non\\
   H_{23} &:=& d^{-3/2}\hn^{-1}\sjn \sum_{k=1}^{n_j} \ex\{{(|X\uj|+|\m\uj|)}\,|\tX\ujk|\}\ex|Z\ujk|; \non\\    
   H_{24}  &:=& d^{-3/2}\hn^{-1}\sjn \sum_{k=1}^{n_j} \ex\{{(|X\uj|+|\m\uj|)}\,|\tX\ujk|\}\ex|Z\uj|,\non
\ena
and then setting
\eqa
   H_0 &:=& d^{-1/2}\hn^{-1}\sjn \ex{|X\uj|}; \non\\
   H_{1} &:=&  d^{-1}\hn^{-1}\sjn \sum_{k=1}^{n_j} \ex\{{(|X\uj|+|\m\uj|)}\,|\tX\ujk|\};\label{H-defs}\\
   H_{2} &:=&  \max_{1\le l\le 4}H_{2l}.\non
\ena
We also assume that
\eq\label{Zhat-def}
   \ex\{|Z\uj|^2\} \Le d\hn;\qquad \ex\{|Z\ujk|^2\} \Le d\hn, \qquad \mbox{for all}\ 1\le j\le n;\,1\le k\le n_j.
\en
In view of the definitions of $\hn$ and~$V$, $H_0$ and~$H_1$ can be expected to
be of moderate size in many applications, $H_2$  can be expected to grow with the size of a
typical neighbourhood of a vertex~$j$, and the assumption~\Ref{Zhat-def} can be expected to be satisfied.
The various $d$-factors are designed to offset any automatic dimension dependence in the corresponding
quantities, but their choice plays no essential part in the bounds given below.

We now make a smoothness assumption on the distributions of $W\uj$ and~$W\ujk$ that is key for approximation
in total variation. We
assume that, for each $1\le j \le n$, $1\le k\le n_j$ and $1\le i\le d$, we have
\eq\label{dtv-assn}
  \begin{array}{l}
   \dtv\bigl(\law(W\uj + e\uii \giv X\uj,Z\uj),
                 \law(W\uj \giv X\uj,Z\uj) \bigr) \Le {\e_W}\ {\rm a.s.}; \\[1ex]
   \dtv\bigl(\law(W\ujk + e\uii \giv X\uj,\tX\ujk,Z\ujk),
                 \law(W\ujk \giv X\uj,\tX\ujk,Z\ujk) \bigr) \Le {\e_W}\ {\rm a.s.},
  \end{array}
\en
for some ${\e_W} < 1$.  
Of course, for the bounds that we shall prove, we shall want~${\e_W}$ to be suitably small.
This assumption is clearly useful in establishing Condition~(a) of Theorem~\ref{ADB-DN-approx-thm},
but is also used throughout the treatment of $|\ex\{\ABA_\hn h(W)\}I[|W-\m| \le {\hn}\d]|$.

\begin{theorem}\label{DN-approx}
Let $W := \sum_{j=1}^n X^{(j)}$ be decomposed as above, with~\Ref{Zhat-def} satisfied,
and suppose that~$V$ is positive definite.  Then there exist constants 
$C_{\ref{DN-approx}}$ and~$n_{\ref{DN-approx}}$, depending continuously on the condition number~$\r(V)$, such that
\eqs
   \lefteqn{\dtv(\law(W),\DN_d(\m,V) }\\
  &&\Le C_{\ref{DN-approx}}d^3\log\hn\{(d+H_2)\e_{W} + (d+H_0+H_2+\hn^{-1/2}H_1)\hn^{-1/2}
               + (\ch_1 + \ch_2 + \ch_3)\},
\ens
for all ${\hn} \ge n_{\ref{DN-approx}}$. 
\end{theorem}

Our second setting is somewhat more general. We suppose that~$W$ results from integrating the marks of a 
marked point process with respect to its ground process. We assume that the carrier space~$\G$ of 
the ground point process~$\Xi$ is a locally compact second countable Hausdorff topological  space (Kallenberg~(1983, p.~11)),
with Borel $\s$-field $\caB(\G)$.  Let $\tG:=\G\times \Z^d$, and equip it with the product Borel 
$\s$-field $\caB(\tG)=\caB(\G)\times \caB(\Z^d)$.
We use~$\caH$ to denote the space of all locally finite non-negative integer valued measures $\xi$ on~$\tG$ such 
that $\xi(\{\a\}\times \Z^d)\le 1$ for all $\a\in\G$. The space~$\caH$ is endowed with the $\s$-field 
$\caB(\caH)$ generated by the vague topology (Kallenberg~(1983, p.~169)). 
A {\it marked point process\/} $\tXi$ is a measurable mapping from
$(\Omega,\caF,\pr)$ to $(\caH,\caB(\caH))$ (Kallenberg~(2017, p.~49)). The induced simple point process 
$\Xi(\cdot):=\tXi(\cdot\times\Z^d)$ is called the {\it ground process\/} (Daley \& Vere-Jones~(2008, p.~3)) or 
projection  (Kallenberg~(2017, p.~17)) of 
the marked point process~$\tXi$. We define $X\uja=yI[\tXi(\{(\a,y)\})=1]$ 
to represent the {\it mark\/} of $\tXi$ at~$\a$. We assume that the ground process~$\Xi$ is locally finite,
 with mean measure~$\nu$.

Let $\{D_\a,\a\in\G\}$ be a class of neighbourhoods such that, for each $\a\in\G$, $D_\a\in\caB(\G)$ is a Borel 
set containing~$\a$ and such that $D=\{(\a,\b)\colon \b\in D_\a, \a\in\G\}$ is  a measurable subset of 
the product space
$\G^2:=\G\times\G$ with the product Borel $\s$-field $\caB(\G)\times \caB(\G)$. For the neighbourhoods 
$\{D_\a,\a\in\G\}$, one can easily adapt the proof in Chen \& Xia~(2004) to show that the 
mapping $(\a,\xi)\mapsto (\a,\xi|_{D_\a\times\Z^d})$
is a measurable mapping from $(\G\times\caH,\caB(\G)\times\caB(\caH))$ into
itself, where $\xi|_{D_\a\times\Z^d}$ is the restriction of $\xi\in\caH$ to $D_\a\times\Z^d$ (Kallenberg~(1983, p.~12)). 

Our goal is to establish the accuracy of discrete normal approximation to~$W = \intG X\uja \Xi(d\a)$. 
When $\G=\{1,\dots,n\}$ and~$\Xi$ is the counting measure on~$\G$, $W$ reduces to the sum in the previous setting,
{so the bound in Theorem~\ref{DN-approx} is a corollary of that in Theorem~\ref{DN-approx-cont}.
However, if there is dependence between $\Xi$ and~$X$, then there is significant difference between the two settings. 
For the latter setting, it is necessary to introduce {extra machinery,
including} the first and second order Palm distributions (Kallenberg~(1983, p.~83 and p.~103)), 
to tackle the problem.}

For convenience, we use $\pr_\a$, $\ex_\a$ and $\law_\a$ to stand for the conditional probability, conditional 
expectation and conditional distribution given $\{\Xi(\{\a\})= 1\}$ respectively.
It is a routine exercise (Kallenberg~(1983, pp.~83--84)) to show that $\ex_{\a}$ satisfies 
$$
          \ex\intG f(X\uja,\a)\Xi(d\a)=\intG \ex_\a f(X\uja,\a)\nu(d\a)
$$
for all non-negative functions~$f$ on $(\Z^d\times\G,\caB(\Z^d)\times\caB(\G))$.
Similarly, for $\a\ne \b$, we use $\pr_{\a\b}$, $\ex_{\a\b}$ and $\law_{\a\b}$ to stand for the conditional probability, 
conditional expectation and conditional distribution given $\{\Xi(\{\a\})= 1\}\cap \{\Xi(\{\b\})= 1\}$ respectively. 
Writing $\nu_2(d\a,d\b)=\ex(\Xi(d\a)\Xi(d\b))$ for $\a\ne \b$, we can also show that $\ex_{\a\b}$ satisfies,  
\eq
     \ex\int_{\a,\b\in\G,\a\ne\b} f(X\uja,X\ukb,\a,\b)\Xi(d\b)\Xi(d\a)
       \Eq \ex\int_{\a,\b\in\G,\a\ne\b} \ex_{\a\b} f(X\uja,X\ukb,\a,\b)\nu_2(d\a,d\b),\label{secondPalmid}
\en
for any non-negative measurable function~$f$ on 
$(\Z^d\times\Z^d\times\G\times\G,\caB(\Z^d)\times\caB(\Z^d)\times\caB(\G)\times\caB(\G))$.
To avoid unnecessary complexity and to keep our notation consistent, we write $\pr_{\a\a}=\pr_\a$, $\ex_{\a\a}=\ex_\a$, 
$\law_{\a\a}=\law_{\a}$ and $\nu_2(d\a,d\a)=\nu(d\a)$ so that \Ref{secondPalmid} can be extended to
$$
      \ex\int_{\a,\b\in\G} f(X\uja,X\ukb,\a,\b)\Xi(d\b)\Xi(d\a)
        \Eq \ex\int_{\a,\b\in\G} \ex_{\a\b} f(X\uja,X\ukb,\a,\b)\nu_2(d\a,d\b).
$$
We set $\nu_\a(d\b)=\nu_2(d\a,d\b)/\nu(d\a)$.

As in the previous setting, we  assume that there are decompositions of the following form:
\begin{description}
  \item {(a')}
  For each $\a\in\G$, we can write $W = W\uja + Z\uja$, where~$W\uja \in \Z^d$ is only weakly dependent on~$X\uja$; 
  \item {(b')}
  For each $\a\in\G$, we can write $Z\uja = \int_{D_\a} \tX\ujkab \Xi(d\b)$, with $\tX\ujkab \in \Z^d$,
   and then, for each $\b\in D_\a$,
   we can write $W\uja = W\ujkab + Z\ujkab$, where~$W\ujkab \in \Z^d$ is only weakly dependent on~$(X\uja,\tX\ujkab)$. 
   In particular, we take $\tX\ujkaa=X\uja$, $Z\ujkaa=Z\uja$ and $W\ujkaa=W\uja$.
\end{description}
Next, we set $\m\uja := \ex_\a X\uja$ and $\mu:=\intG\m\uja\nu(d\a)$, and define
\eqa
    \ch_{12\a}'  &:=&  \ex_\a\bigl|\ex_\a(|X\uja| \giv W\uja) - \ex_\a |X\uja| \bigr|;
              \non\\
   \ch_{13\a}'  &:=&  \ex_\a\bigl|\ex_\a(|X\uja|_1 \giv W\uja) - \ex_\a |X\uja|_1 \bigr|;
              \label{Chi-j-cont}\\ 
   \ch_{2\a\b}' &:=&  \sum_{i=1}^d \sum_{l=1}^d \ex_{\a\b}\bigl|\ex_{\a\b}\{|X_i\uja |\,|\tX_l\ujkab| \giv W\ujkab\}
                    - \ex_{\a\b}\{|X_i\uja|\,|\tX_l\ujkab|\}\bigr|\non\\
                    &&  \ \ \ +\sum_{i=1}^d \sum_{l=1}^d \ex_{\a\b}\bigl|\ex_{\a\b}\{X_i\uja \tX_l\ujkab \giv W\ujkab\}
                    - \ex_{\a\b}\{X_i\uja\tX_l\ujkab\}\bigr|;\non\\
  \ch_{2\a\b}'' &:=& \sum_{i=1}^d \sum_{l=1}^d |\m_i\uja |\,\ex_\b\bigl|\ex_\b\{|\tX_l\ujkab| \giv W\ujkab\}
                    - \ex_\b\{|\tX_l\ujkab|\}\bigr|\non\\
                    &&\ \ \ + \sum_{i=1}^d \sum_{l=1}^d |\m_i\uja |\,\ex_\b\bigl|\ex_\b\{\tX_l\ujkab \giv W\ujkab\}
                    - \ex_\b\{\tX_l\ujkab\}\bigr|. \non
\ena
{Note that the quantities $\ch_{2\a\b}'$ and~$\ch_{2\a\b}''$ are more complicated than their counterparts
$\ch_{2jk}$ in the earlier setting, to allow for possible dependence between the ground process and the
marks.}
Then let
\eqa
   \ch_{11}'&:=& (d\hn)^{-1/2} \intG \ex_\a|\ex_\a(X\uja \giv W\uja) - \m\uja|\nu(d\a); \non\\
   \ch_{12}'  &:=& (d\hn)^{-1/2} \intG \ch'_{12\a}\nu(d\a); \qquad\qquad \ch'_{13}  \Eq d^{-1}\hn^{-1/2} \intG \ch'_{13\a}\nu(d\a);
              \non\\          
  \ch_1' &:=& \max_{1\le l\le 3}\ch'_{1l}\non\\
  \ch_{2}' &:=& d^{-3}\hn^{-1}\left(\intG \int_{D_\a} \ch'_{2\a\b}\nu_2(d\a,d\b)+\intG \int_{D_\a} \ch''_{2\a\b}\nu(d\b)\nu(d\a)\right);  \label{RR-csts-cont}\\
  \ch_3' &:=& d^{-1}\hn^{-1}\intG \ex_\a\{|\ex_\a(X\uja\giv W\uja)-\m\uja|\,|W\uja-\m|\}\nu(d\a).\non
\ena

As in the discrete sum, we assume that $\ex|X\uja|^3 < \infty$ $\nu$-a.s. Recalling
$$
   \m = \ex W \Eq \intG \m\uja\nu(d\a), 
$$
we define
\eq
   V \Def \cov(W);\quad \hn := \lceil d^{-1}\tr V \rceil; \quad 
            \cc \Def \hn^{-1}\m;\quad \S \Def \hn^{-1}V. \label{hn-def-cont}
\en
We next introduce some moment sums by defining
\eqa
   H_{21}' &:=& d^{-3/2}\hn^{-1}\intG \{\ex_\a\{|X\uja|\,|Z\uja|^2\}+|\m\uja|\,\ex\{|Z\uja|^2\}\}\nu(d\a);               \non \\
   H_{22}' &:=& d^{-3/2}\hn^{-1}\intG \int_{D_\a} \ex_{\a\b}\{|X\uja|\,|\tX\ujkab|\,|Z\ujkab|\}\nu_2(d\a,d\b)  \non\\
                &&+d^{-3/2}\hn^{-1}\intG \int_{D_\a}|\m\uja| \ex_{\b}\{|\tX\ujkab|\,|Z\ujkab|\}\nu(d\b)\nu(d\a);   \non\\
  H_{23}' &:=& d^{-3/2}\hn^{-1}\intG \int_{D_\a} \ex_{\a\b}\{|X\uja|\,|\tX\ujkab|\,\}\ex\{|Z\ujkab|\}\nu_2(d\a,d\b)  \non\\
                &&+d^{-3/2}\hn^{-1}\intG \int_{D_\a}|\m\uja| \ex_{\b}\{|\tX\ujkab|\,\}\ex\{|Z\ujkab|\}\nu(d\b)\nu(d\a);   \non\\
 H_{24}' &:=& d^{-3/2}\hn^{-1}\intG \int_{D_\a} \ex_{\a\b}\{|X\uja|\,|\tX\ujkab|\,\}\ex\{|Z\uja|\}\nu_2(d\a,d\b)  \non\\
                &&+d^{-3/2}\hn^{-1}\intG \int_{D_\a}|\m\uja| \ex_{\b}\{|\tX\ujkab|\,\}\ex\{|Z\uja|\}\nu(d\b)\nu(d\a),   \non\\
\ena
{noting the extra complication in $H_{22}'$, $H_{23}'$ and~$H_{24}'$ as compared with $H_{22}$, $H_{23}$ and~$H_{24}$,}
and then setting
\eqa
   H_0' &:=& d^{-1/2}\hn^{-1}\intG \ex_\a|X\uja|\nu(d\a); \non\\
   H_{1}' &:=&  d^{-1}\hn^{-1}\intG \int_{D_\a} \ex_{\a\b}\{|X\uja|\,|\tX\ujkab|\}\nu_2(d\a,d\b)\label{H-defs-cont}\\
              &&+d^{-1}\hn^{-1}\intG |\m\uja|\int_{D_\a} \ex_\b\{|\tX\ujkab|\}\nu(d\b)\nu(d\a);\non\\
   H_{2}' &:=&  \max_{1\le l\le 4}H'_{2l}.\non
\ena
As a consequence of the dependence between the marks and the ground process, 
{the analogue of~\Ref{Zhat-def} is more involved:} we need to assume that 
there exists a constant $\ccc\ge 1$ such that 
\eqa
  &&\ex_\a(|W-\m|^2)\Le \ccc d\hn \ \nu-{\rm a.s.},\ \ex_{\a\b}(|W-\m|^2)\Le \ccc d\hn\ \nu_2-{\rm a.s.},\non\\
  &&\ex\{|Z\uja|^2\}\Le d\hn,\ \ex_\a\{|Z\uja|^2\}\Le d\hn\ \nu-{\rm a.s.}, \label{Zhat-def-cont} \\
  &&\max\{\ex_{\a\b}\{|Z\uja|^2\},\ex_\b\{|Z\uja|^2\},\ \ex_\b\{|Z\ujkab|^2\},
          \ex_{\a\b}\{|Z\ujkab|^2\}\}\Le d\hn \ \nu_2-{\rm a.s.}. \non
\ena
{The analogue of~\Ref{dtv-assn} is even more involved.  First,}
for $1\le i\le d$, $\nu$-a.s. in $\a$ and $\nu_2$-a.s. in $\a,\b$,
we need to find $\e_{W}'< 1$ such that
\eq\label{dtv-assn-cont}
  \begin{array}{l}
  \dtv\bigl(\law(W\uja + e\uii \giv X\uja,Z\uja),
                 \law(W\uja \giv X\uja,Z\uja) \bigr) \Le \e_{W}'\ {\rm a.s.}; \\[1ex]
   \dtv\bigl(\law_\a(W\uja + e\uii \giv X\uja,Z\uja),
                 \law_\a(W\uja \giv X\uja,Z\uja) \bigr) \Le \e_{W}'\ {\rm a.s.}; \\[1ex]
   \dtv\bigl(\law(W\ujkab + e\uii \giv X\uja,\tX\ujkab,Z\ujkab),
                 \law(W\ujkab \giv X\uja,\tX\ujkab,Z\ujkab) \bigr) \Le \e_{W}'\ {\rm a.s.},\\[1ex]
 \dtv\bigl(\law_\b(W\ujkab + e\uii \giv X\uja,\tX\ujkab,Z\ujkab),
                 \law_\b(W\ujkab \giv X\uja,\tX\ujkab,Z\ujkab) \bigr) \Le \e_{W}'\ {\rm a.s.},\\[1ex]
 \dtv\bigl(\law_{\a\b}(W\ujkab + e\uii \giv X\uja,\tX\ujkab,Z\ujkab),
                 \law_{\a\b}(W\ujkab \giv X\uja,\tX\ujkab,Z\ujkab) \bigr) \Le \e_{W}'\ {\rm a.s.}\,.
  \end{array}
\en
{We then also need to find $\e_{W}''<1$ such that}
\eqa 
&&\dtv\bigl(\law_\a(W\uja),\law(W\uja)\bigr)\Le \e_{W}''\ \nu-{\rm a.s.}, \non\\
&&\dtv\bigl(\law_{\a\b}(W\ujkab),\law(W\ujkab)\bigr)\Le \e_{W}''\ \nu_2-{\rm a.s.},\non\\
&&\dtv\bigl(\law_\b(W\ujkab),\law(W\ujkab)\bigr)\}\Le \e_{W}''\  \nu_2-{\rm a.s.} . \label{dtvmissed1}
\ena
{Finally, we need a bound controlling the difference between some conditional and
unconditional expectations: we need to find} $\e_{W}'''<1$ such that
\eqa
&& |\ex_\a(W\uja)-\ex(W\uja)|\Le \e_{W}'''\ \nu-{\rm a.s.}\,.\label{dtvmissed2}
\ena
{Fortunately,
under many circumstances (see Barbour \& Xia~(2006)), 
both $\e_W''$ and $\e_W'''$ can be reduced to~$0$, as is the case in Example~\ref{max-pts}.}

\begin{theorem}\label{DN-approx-cont}
Let $W := \intG X\uja\Xi(d\a)$ be decomposed as above, such that~\Ref{Zhat-def-cont} is satisfied,
and suppose that~$V$ is positive definite.  Then there exist constants 
$C_{\ref{DN-approx-cont}}$ and~$n_{\ref{DN-approx-cont}}$, depending continuously on $\r(V)$, such that
\eqs
   \lefteqn{\dtv(\law(W),\DN_d(\m,V) }\\
  &&\Le C_{\ref{DN-approx-cont}}d^3\log\hn\{(d+H_2')\e_{W}' + (d+H_0'+H_2'+\hn^{-1/2}H_1')\hn^{-1/2}
               + (\ch_1' + \ch_2' + \ch_3')\\
  &&\ \ \ \ \ \ \ \ \ \ \ \ \ \ \ \ \ \ \ \ \ \ +\e_{W}''(d^{-2}m^{1/2}H_0'+d^{-1}H_1')+d^{-3/2}H_0'\e_{W}'''\},
\ens
for all ${\hn} \ge n_{\ref{DN-approx-cont}}$. 
\end{theorem}

\section{Intersection graph dependence}\label{independence}
\setcounter{equation}{0}

In this section, we consider sums $W := \sjn X\uj$ of random vectors~$X\uj$ that are determined 
by the values of an underlying collection of independent
random elements $(Y_{{i}},\,1\le {i}\le M)$;  we assume that $X\uj := X\uj((Y_{{i}},\,{i}\in M_j))$, for some
subset $M_j {\subset} [M] := \{1,2,\ldots,M\}$.  The subsets~$M_j$ induce an intersection 
graph~$G$ on~$[n]$, in which
there is an edge between $j$ and~$k\neq j$, $j\sim k$, exactly when $M_j\cap M_k \neq \emptyset$;
we denote by 
$N_j := \{k \in [n]\setminus \{j\}\colon k\sim j\}$ the neighbourhood of~$j$ in~$G$.  With
this definition, $X\uj$ is independent of $(X\uk\colon k\in [n]\setminus(\{j\}\cup N_j))$, 
and the graph~$G$ is a dependency graph in the sense of Baldi \& Rinott~(1989).

In this setting, there is a natural way to define $W\uj$ and~$W\ujk$.
For each $j\in [n]$, we define $Z\uj := X\uj + \sum_{k\sim j} X\uk$ and $W\uj := W - Z\uj$, 
noting that $W\uj$ and~$X\uj$ are independent, so that $\tX\ujk = X\uk$, $k \in N_j \cup\{j\}$, 
and $\ch_1=\ch_3=0$.  Then, for $k = j$, $W\ujk = W\uj$ and $Z\ujk = 0$; otherwise,
for $j\neq k \in [n]$ such that $j\sim k$, we define $W\ujk := \sum_{l \notin N_j \cup N_k}X\ul$
and $Z\ujk := W\uj - W\ujk$; note that $W\ujk$ and the pair $(X\uj,X\uk)$ are independent,
so that $\ch_2=0$ also.   If we also impose some uniformity, by supposing that
\eq\label{third-moments}
   1 \Le \max_{1\le j\le n}d^{-3/2}\ex |X\uj|^3 \ =:\ \g\ <\ \infty,
\en
then we have the following corollary of Theorem~\ref{DN-approx}.

\begin{corollary}\label{cor1}
Suppose that the above assumptions are satisfied.  Define
\[
   D_j \Def |N_j|, \quad\mbox{and}\quad \bdt \Def \hn^{-1}\sjn (D_j+1)^2.
\]
Then 
\[
   \dtv(\law(W),\DN_d(\m,V) \Le C_{\ref{DN-approx}}d^3\log\hn(\hn^{-1/2} + \e_{W})\{d + 3\g\bdt\},
\]
for all $n \ge n_{\ref{DN-approx}}$. 
\end{corollary}

\begin{proof}
All that is needed is to observe that $H_1 \le \half(H_0+H_2)$, and that
$\max\{H_0,H_2\} \le \g\bdt$. \ep
\end{proof}

The main difficulty in applying the bounds in Theorem~\ref{DN-approx} and Corollary~\ref{cor1} is
putting a value to~$\e_{W}$.  This can nonetheless often be dealt with, provided that enough of the
underlying random variables~$(Y_l,\,l\in[M])$ each influence rather few of the~$X\uj$.  The next
theorem gives a way of exploiting this.

Given any $l \in [M]$, define $L_l := \{j\in[n]\colon M_j \ni l\}$,
and $S_l := \sum_{j \in L_l} X\uj$; write $\gg_l := \s(Y_{l'},\,l'\in [M]\setminus \{l\})$,
and define
\[
    d_l\uii(Y) \Def \dtv(\law(S_l \giv \gg_l), \law(S_l + e\uii \giv \gg_l)),\quad 1\le i\le d.
\]
Given any $j \ne k \in [n]$ such that $j\sim k$, define
\[
   M\ujk \Def \bigcup_{j' \in N_j \cup N_k} M_{j'},
\]
and find $l_1<l_2<\cdots<l_s \in [M]\setminus M\ujk$ such that
$L_{l_r} \cap L_{l_{r'}} = \emptyset$ for all $1 \le r < r' \le s$.  Then the vectors 
$S_{l_1},\ldots,S_{l_s}$ are conditionally 
independent, given $\ff\ujk := \s\bigl(Y_l,\, l\notin \{l_1,\ldots,l_s\}\bigr)$.
Write 
\[
    D_i\ujk(Y) \Def \sum_{r=1}^s (1 - d_{l_r}\uii(Y)).
\]

\begin{theorem}\label{smoothness}
Suppose that, for $j \ne k \in [n]$ such that $j\sim k$, we can find~$s$ and
$l_1<l_2<\cdots<l_s \in [M]\setminus M\ujk$ such that the sets~$L_{l_r}$, $1\le r\le s$,
are disjoint, and such that
\[
   \pr[D_i\ujk(Y) \le T] \Le \h.
\]
Then 
\[
  \dtv\bigl(\law(W\ujk + e\uii \giv \{X\ur,\,r\in N_j \cup N_k\}),
                   \law(W\ujk \giv \{X\ur,\,r\in N_j \cup N_k\}) \bigr) \Le {\Bl\frac2{\pi T}\Br^{1/2}} + \h.
\]                              
\end{theorem}
 
\begin{proof}
Writing $U\ujk := \sum_{r=1}^s S_{l_r}$, we have
\eqs
   \lefteqn{ \dtv\bigl(\law(W\ujk + e\uii \giv \{X\ur,\,r\in N_j \cup N_k\}),
                              \law(W\ujk \giv \{X\ur,\,r\in N_j \cup N_k\}) \bigr)} \\
  &&\Le \ex\bigl\{\dtv\bigl(\law(W\ujk + e\uii \giv \ff\ujk), \law(W\ujk \giv \ff\ujk) \bigr)  \bigr\} \\
  &&\Le \ex\bigl\{\dtv\bigl(\law(U\ujk + e\uii \giv \ff\ujk), \law(U\ujk \giv \ff\ujk) \bigr)  \bigr\} .                           
\ens
Now, by the Mineka coupling argument (Lindvall, 2002, Section~II.14),
\[
   \dtv\bigl(\law(U\ujk + e\uii \giv \ff\ujk), \law(U\ujk \giv \ff\ujk) \bigr) 
               \Le \Bl\frac2{\pi D_i\ujk(Y)}\Br^{1/2},
\] 
{where the constant comes from Mattner  \& Roos~(2007), Corollary~1.6,}
and the theorem follows. \ep
\end{proof}

\section{Examples}\label{examples}
\setcounter{equation}{0}

{In this section, we demonstrate that Theorems~\ref{DN-approx} and \ref{DN-approx-cont} can be easily 
applied in a range of situations. The first three examples are discrete sums, and Theorem~\ref{DN-approx} 
can be invoked. In the last example, we need Theorem~\ref{DN-approx-cont}.}

\subsection{Graph colouring}\label{colouring}
As a first example, suppose that the vertices in a graph~$G := ([M],E)$  are coloured
independently, with colour~$i$ being chosen with probability~$\p_i$, $1\le i\le d$.  
Let $Y_l$ be the colour of vertex~$l$, and
let~$W_i$ denote the number of edges of~$G$ that connect two vertices of colour~$i$;
write $W := (W_1,\ldots,W_d)^{{T}}$.  Then~$n := |E|$
is the number of edges in~$G$, and, for $j,k\in E$, $j \sim k$ if $j$ and~$k$ share a common vertex.
For $l\in [M]$, let $\d_l$ denote the degree of $l$ in~$G$; then, for $j := \{l,l'\} \in E$,
$D_j := |N_j| = {\d_l + \d_{l'}}$.   Define 
\[
    \tbd \Def n^{-1} \sum_{j\in E}D_j,\quad \tbdt \Def n^{-1} \sum_{j\in E}D_j^2.
\]
Then it is easy to compute
\eqs
   \m_i &=& \ex W_i \Eq n\p_i^2;\quad V_{ii} \Eq \var W_i \Eq n\{\pi_i^2(1-\p_i^2) + \tbd \p_i^3(1-\p_i)\};\\
    V_{ii'} &=& \cov(W_i,W_{i'}) \Eq -n\pi_i^2 \pi_{i'}^2(1+\tbd), \quad i\ne i'.
\ens
Thus $\tr V = nd\{c_1 + \tbd c_2\}$, where 
\[
     c_1 \Def d^{-1}\sum_{i=1}^d \pi_i^2(1-\p_i^2);\quad c_2 \Def d^{-1}\sum_{i=1}^d \pi_i^3(1-\p_i),
\]
so that we take $\hn := \lceil n(c_1 + \tbd c_2)\rceil$ in Corollary~\ref{cor1}.  We can clearly take 
$\g = 1$ also, and, for fixed~$d$ and $\p_1,\ldots,\p_d$, this yields a bound  
\[
     \dtv(\law(W),\DN_d(\m,V)) \Eq O\bigl\{(\hn^{-1/2} + \e_{W})\log \hn\, (1 + \tbdt/\tbd)\bigr\},
\]
which relies on having a reasonable bound for~$\e_{W}$.  

In order to apply Theorem~\ref{smoothness}, for each $j\sim k \in E$, we want
first to find~$s$ and $l_1,l_2,\ldots,l_s \in [M]\setminus\{M_j \cup M_k\}$
such that the sets~$L_{l_1},\ldots,L_{l_s}$ are disjoint.  Now 
$L_l = \bigl\{\{l,l'\}\colon l'\in[M], \{l,l'\} \in E\bigr\}$, so that $|L_l| = \d_l$,
and $L_l \cap L_{l'} \ne \emptyset$ exactly when $\{l,l'\} \in E$. Thus we
need to find a set of vertices $l_1,\ldots,l_s$ subtending no edges of~$G$ 
({\it independent\/} in the graph theoretical sense).  Letting~${\d^\ast}:= \max_l \d_l$, we note that
$|[M]\setminus\{M_j \cup M_k\}| \ge M-3{\d^\ast}$, and that we can thus always take
$s \ge s(M,{\d^\ast}) := \lfloor M/({\d^\ast}+1) \rfloor - 3$.  

The next step is to bound $d_l\uii(Y) = d_l\uii(\{Y_{l'},\,\{l,l'\} \in E\})$, for each $1\le i\le d$ 
and for any~$l$. To do so, let $R_{il} :=  \sum_{l'\colon\{l,l'\}\in E} I[Y_{l'}=i]$ be the number
of neighbours of~$l$ in~$G$ that have colour~$i$.  Then~$S_l$ takes one of the values
$R_{1l}e\ui,\ldots,R_{dl}e\ud \in \Z^d$, with conditional probabilities $\p_1,\ldots,\p_d$.
Hence, if $R_{il} = 1$ and $R_{i'l} = 0$ for some $i' \ne i$, then $S_l = e\uii$ with 
conditional probability~$\p_i$, and $S_l = 0$ with conditional probability at least~$\p_{i'}$,
giving $d_l\uii \le 1 - \min\{\p_i,\p_{i'}\}$.  Hence, for any $i' \ne i$, we have
\eqs
    \sum_{r=1}^s (1 - d_{l_r}\uii) &\ge&  \sum_{r=1}^s  I[R_{i,l_r} = 1, R_{i',l_r} = 0] \min\{\p_i,\p_{i'}\} \\
                                   &=:& \min\{\p_i,\p_{i'}\} \hR(i,i').      
\ens
Now, if $\d_l = t$,
\[
   \pr[R_{il} = 1, R_{i'l} = 0] \Eq h(t,i,i') \Def t\p_i (1-\pi_i-\p_{i'})^{t-1},
\]
giving $\ex \hR(i,i') = \sum_{r=1}^s h(\d_{l_r},i,i') \ge s \hmin(i,i')$,
where $\hmin(i,i') := \min_{1\le t\le {\d^\ast}} h(t,i,i')$.  Then, since
the events $\{R_{i,l_r} = 1, R_{i',l_r} = 0\}$ and $\{R_{i,l_{r'}} = 1, R_{i',l_{r'}} = 0\}$
are independent unless there is a path of length~$2$ connecting $l_r$ and~$l_{r'}$, we have
\[
  \var(\hR(i,i')) \Le \ex \hR(i,i') (1 +  {\d^\ast}({\d^\ast}-1)).
\]
Hence, by Chebyshev's inequality,
\[
   \pr[\hR(i,i') \le \half s(M,{\d^\ast}) \hmin(i,i')] \Le \frac{4(1 +  {\d^\ast}({\d^\ast}-1))}{\ex \hR(i,i')}
                    \Le \frac{4{\d^\ast}^2}{s(M,{\d^\ast})\hmin(i,i')}.
\]
Thus we can take 
\[
   T \Eq \half s(M,{\d^\ast}) \min\{\p_i,\p_{i'}\} \hmin(i,i')\quad\mbox{and}\quad \h \Eq \frac{4{\d^\ast}^2}{s(M,{\d^\ast})\hmin(i,i')}
\]
in Theorem~\ref{smoothness}.  If, as $M \to \infty$, $n \ge cM$ for some $c>0$ and~${\d^\ast}$ remains bounded,
with the colour probabilities remaining constant, this gives $\e_{W} = O(M^{-1/2})$,
and so
\[
     \dtv(\law(W),\DN_d(\m,V)) \Eq O\bigl\{ M^{-1/2}\log M \bigr\}.
\]
The order in~$M$ is the same as is obtained,
in the context of {$\d^\ast$}-regular graphs and using the convex sets metric, by Rinott \& Rotar~(1996).

Note that, if most of the degrees in~$G$ become large as~$M$ increases,  $\hmin(i,i')$ may well converge 
to zero too fast for the bound on~$\e_{W}$ to be useful, and more sophisticated arguments would be needed. 
Note also that, for $d=2$, $h(t,i,i') = h(t,1,2) = 0$ for all $t\ne 1$, because $\p_1+\p_2=1$,
and we obtain no bound on~$\e_{W}$ in this way.  
Indeed, if~$G$ is an {$\d^\ast$}-regular graph and $d=2$, $\e_{W}$ is {\it not\/} small, since the distribution of~$W$ is
concentrated on a sub-lattice of~$\Z^2$ if ${\d^\ast} \ge 2$, and $\law(W)$ is no longer close to $DN_d(\m,V)$ 
in total variation; see \BLX~(2018b, Section 4.2.1).

The problem can be modified, by only counting a random subset of monochrome edges.  Let $(\tY_j,\,j\in E)$
be independent $\Be(p)$ random variables, and define $\tW_i := \sum_{j\in E}\tY_j X\uj$, where~$X\uj$ is as
before.  Then 
\eqs  
    \tm &:=& \ex \tW \Eq p\m; \quad  \tV_{ii} \Def \var(\tW_{ii}) \Eq n\{p\p_i^2(1-p\p_i^2) + \tbd p^2\p_i^3(1-\p_i)\};\\
  \tV_{ii'} &:=& \cov(\tW_i,\tW_i') \Eq -np^2\p_i^2\p_{i'}^2(1+\tbd), 
\ens
giving $\hn = \lceil nd\{pc_1 + p(1-p)c_1' + p^2 \tbd c_2\}\rceil$, where $c_1' := d^{-1}\sum_{i=1}^d \p_i^4$.  As before,
for fixed~$d$ and $\p_1,\ldots,\p_d$, this yields
\[
     \dtv(\law(\tW),\DN_d(\tm,\tV)) \Eq O\bigl\{(\hn^{-1/2} + \te_n)(1 + \tbdt{/\tbd})\log \hn \bigr\}.
\]
However, the quantity~$\te_n$ is rather easier to bound than~$\e_{W}$, since we can take the independent random
variables~$(\tY_j,\,j\in E)$ to use in Theorem~\ref{smoothness}, each of which influences only the
corresponding~$X\uj$.  Conditional on the colours, $\gg := \s(Y_l,\,l\in[M])$, we have
\[
     \tW_i \Eq \sum_{j = \{l,l'\}\in E} I[Y_l=Y_{l'}=i]\, \tY_j \ \sim\ \Bi(W_i,p),
\]
with $\tW_1,\ldots,\tW_d$ conditionally independent,
and hence $\dtv(\law(\tW\giv\gg),\law(\tW+e\uii\giv\gg)) \le 1/\sqrt{pW_i}$.
Using the moments of~$W$ calculated above, it follows easily that $\te_n = O(\{np\}^{-1/2})$,
giving
\eq\label{GC-order-2}
     \dtv(\law(\tW),\DN_d(\tm,\tV)) \Eq O\bigl\{(Mp\tbd)^{-1/2}(1 + \tbdt{/\tbd})\log M \bigr\}.
\en

The apparent order in~$\tbd$ is misleading here. 
If~$\tbd$ is large, the covariance matrix~$V$ is ill conditioned, since 
$\tr V \asymp n\tbd \asymp M\tbd^2$,
whereas $\var\bigl\{\sum_{i=1}^d \p_i^{-1}W_i\bigr\} = n(d-1) \asymp  M\tbd$. Thus the condition number~$\r(V)$
grows like~$\tbd$, and~$\r(V)$ enters the constant~$C_{\ref{DN-approx}}$ implied in the order symbol
in~\Ref{GC-order-2}.    However, if only the joint distribution of, say, $(\tW_1,\ldots,\tW_{d-1})^{{T}}$ is of interest, 
the corresponding covariance matrix then has condition number that is bounded in~$\tbd$, for fixed $d$ 
and~$\p_1,\ldots,\p_d > 0$, and the orders in both $M$ and~$\tbd$ are as in~\Ref{GC-order-2}.

   A more general modification, in the same spirit, it to choose $(\tY_j\uii,\,j\in E, 1\le i\le d)$
to be any independent integer valued random variables, with distributions depending only on~$i$,
and to set $\tW_i := \sum_{j= \{l,l'\}\in E}I[Y_l=Y_{l'}=i]\tY_j\uii$.
Then, if the mean and variance of~$\tY_1\uii$ are {$\tmm\uii$ and~$\tv\uii$,} we have
\eqs  
    \tm_i &:=& \ex \tW_i \Eq n\p_i^2{\tmm\uii}; \quad  \tV_{ii} \Def \var(\tW_{ii}) 
                 \Eq n\p_i^2\bigl\{{v\uii + (\tmm\uii)^2}\{ 1-\p_i^2 + \tbd \p_i(1-\p_i)\} \bigr\};\\
  \tV_{ii'} &:=& \cov(\tW_i,{\tW_{i'}}) \Eq -n\p_i^2\p_{i'}^2 {\tmm\uii \tmm\uid} (1+\tbd), 
\ens
from which the corresponding value of~$\hn$ can be deduced.  As above, it is not difficult
to show that
\[
  \dtv(\law(\tW_i\giv\gg),\law(\tW_i+1\giv\gg)) \Eq O(1/\sqrt{u\uii W_i}),
\]
where $u\uii := 1 - \dtv(\tY_1\uii, \tY_1\uii+1)$, from which it follows that~$\te_n = O((M\tbd)^{-1/2})$.
Hence we find from Corollary~\ref{cor1} that
\eq\label{GC-order-3}
     \dtv(\law(\tW),\DN_d(\tm,\tV)) \Eq O\bigl\{(Mp\tbd)^{-1/2}\g(1 + \tbdt)\log M \bigr\},
\en
where $\g := \max_{1\le i\le d}\ex |\tY_1\uii|^3$.

\subsection{Random geometric graphs}\label{RGG}
Let $M := n^2$ points be distributed uniformly and independently over the 
torus $T_n := [0,n]\times[0,n]$.
For some fixed~$r$, join all pairs of points whose distance apart is less than or equal to~$r$.
This yields a particular example of a random geometric graph;  the book by Penrose~(2003)
discusses much more general models, and gives a comprehensive treatment of their properties.
In this section, we illustrate the application of Theorem~\ref{DN-approx} to counting
induced triangles and $2$-stars;  more complicated examples can be treated in much the
same way.   If the positions of the points are denoted by $(Y_l,\,1\le l\le M)$,
we express our statistic as
\[
    W \Def (W_1,W_2)^{{T}}:\quad W_i \Def \sum_{j \in [M]_3} I[G_j = G\uii],
\]
where $[M]_3$ denotes the set of $3$-subsets of~$[M]$, $G_j := G(Y_{j_1},Y_{j_2},Y_{j_3})$ denotes the 
induced graph on the points $Y_{j_1},Y_{j_2},Y_{j_3}$,  $G\ui$ denotes the triangle and~$G\ut$ 
denotes the $2$-star.

For any~$x \in T_n$, 
the probability that any given point lies in the circle of radius~$r$ around~$x$
is $\p r^2/n^2 =: n^{-2}\hpr$.  Hence
$\pr[G_j = G\uii] = n^{-4}p_r\uii$ is the same for all $j \in [M]_3$, and $p_r\ui,p_r\ut \le \hpr^2$.
The quantities $G_j$ and~$G_k$ are independent unless $j$ and~$k$ have at least two of their
vertices in common, $G_j$ is independent of the set $(G_k\colon j\cap k = \emptyset)$,
and the pair $(G_j,G_{j'})$ is independent of the set $(G_k\colon (j\cup j')\cap k = \emptyset)$.
Using these facts, we can make some computations:
\eqs
  \m &:=& \ex W \Eq \binom{M}3 n^{-4}(p_r\ui,p_r\ut)^T \ \sim\ \third n^2 (p_r\ui,p_r\ut)^T;\\
  V_{ii} &:=& \var W_i \ \sim\  c_{ii}n^2;\quad V_{12} \Def \cov(W_1,W_2)\ \sim\  c_{12}n^2,
\ens
and the matrix $\Bigl(\begin{matrix} c_{11}& c_{12}\\c_{21} & c_{22}\end{matrix}\Bigr)$ is non-singular,
with values involving the geometry of intersections of discs in~$\re^2$.  Thus we can take $\hn = cn^2$ 
for some $c > 0$.
The quantities~$H_0$ and~$H_2$ are then easily bounded:
\eqs
    H_0 &\le&  \frac1{cn^{2}\sqrt2}\binom{M}3 \, 2n^{-4} (p_r\ui + p_r\ut)\ \asymp\ 1;\\
    H_2 &\le&  \frac{c'}{cn^{2}} \binom{M}3 \,  n^{-4} \hpr^2 \{1 + (Mn^{-2}\hpr)^4\} \ \asymp\ 1,
\ens
giving
\[
    \dtv(\law(W),\DN_2(\m,V)) \Eq O((n^{-1} + \e_{W})\log n).
\]
It thus remains to bound~$\e_{W}$.  

To do so, break up $[0,n]^2$ into $9\lfloor n/3r \rfloor^2$
non-overlapping $r\times r$ squares, denoted by $Q_{l,l'} := [(l-1)r,lr)\times [(l'-1),l'r)$.  Then
there can be no triangles or $2$-stars with points in two of the squares in 
$\QQ := (Q_{3l,3l'}, 1\le l,l' \le \lfloor n/3r \rfloor)$,
because points in two of them are more than $2r$ apart.  Consider evaluating 
$\dtv(\law(W+\eii \giv \gg), \law(W \giv \gg))$, much as for Theorem~\ref{smoothness},
where~$\gg$ consists of the positions of all points not in members of~$\QQ$, together with
the {\it numbers\/} of points falling in each member of~$\QQ$.  If~$S_{l,l'}$ denotes the contribution 
resulting from assigning positions to the points in~$Q_{3l,3l'}$, then the random variables
$(S_{l,l'},\,1\le l,l' \le \lfloor n/3r \rfloor)$ are conditionally independent, given~$\gg$.
Let~$N(A)$ denote the number of points falling in the set~$A \subset T_n$.
Then the event~$E_{l,l'}$ that $N(Q_{3l,3l'})=1$,
that the rectangle~$[(3l+0.25)r,(3l+0.5)r)\times [(3l'-1)r,3l'r)$ 
contains two points at a distance between $r/2$ and~$r$ from one another, and that $N(U_{l,l'}) = 3$,
where~$U_{l,l'}$ is the union of $(Q_{r,s},\,l-2 \le r \le l+2, l'-2 \le s \le l'+2)$, 
is such that $\pr[E_{l,l'}] \asymp 1$ as $\nti$, and is the same for all~$l,l'$.  Indeed, we have
\[
   \pr[E_{l,l'}] \Eq \ch \pr[N(U_{1,1}) = 3],
\]
for a constant~$\ch > 0$ that is independent of~$n$ also.  Conditional on~$E_{l,l'}$,  we have
\[
    1 - \dtv\bigl(\law(S_{l,l'} \giv E_{l,l'}),\law(S_{l,l'} + \eii \giv E_{l,l'})\bigr) \Eq u_i ,\quad i=1,2,
\]
for $u_1,u_2 > 0$.  Now the events $(E_{l,l'},\,1\le l,l' \le \lfloor n/3r \rfloor)$ are not independent,
but, except for neighbouring pairs of indices, they are only weakly dependent: for $r,r'$ such that
$\max\{|r-l|,|r'-l'| \ge 2\}$, we have
\eqs
    \pr[E_{l,l'},E_{r,r'}] &=& \ch^2 \pr[\{N(U_{l,l'}) = 3\}\cap \{N(U_{r,r'}) = 3\}] \\
              &=& (\ch \pr[N(U_{1,1}) = 3])^2 \Bi(n^2,25r^2/n^2)\{3\}\,\Bi(n^2-3,25r^2/(n^2-25r^2))\{3\} \\
              &=&  \pr[E_{l,l'}]\pr[E_{r,r'}]\{1 + O(n^{-2})\}.
\ens
Hence
\[
   \ex\Bigl\{\sum_{l,l'} I[E_{l,l'}] \Bigr\} \Eq \lfloor n/3r \rfloor^2 \ch \pr[N(U_{1,1}) = 3];\quad
    \var\Bigl\{\sum_{l,l'} I[E_{l,l'}] \Bigr\} \Eq O(n^2),  
\]
and calculations as for Theorem~\ref{smoothness} now easily yield $\e_{W} = O(n^{-1})$.  Hence it follows that
\[
    \dtv(\law(W),\DN_2(\m,V)) \Eq O(n^{-1}\log n).
\]
 
The asymptotics of~$\e_{W}$ are, however, sensitive to the choice of~$r$:  if $r = r_n \to \infty$,
even logarithmically in~$n$, $\pr[N(U_{1,1}) = 3]$ becomes very small, and the bound on~$\e_{W}$ derived
in this way is no longer useful.

\subsection{Finite Markov chains}\label{MCs}
Let~$(Z_j,\,j\ge0)$ be an irreducible, aperiodic Markov chain on the finite state space $\{0,1,\ldots,d\}$, and set
$X\uj := (I[Z_j=1],\ldots,I[Z_j=d])^{{T}}$.
Let $W_{n} := \sum_{j=1}^n X\uj$ denote the vector of the amounts of time spent in the states $1\le i\le d$
between times $1$ and~$n$.
We are interested in the accuracy of approximating the distribution of~$W_n$ 
by $\DN_d(\m_n,V_n)$, where $\m_n := \ex W_n$ and $V_n := \cov W_n$; translated Poisson approximation
for each component~$W_{in}$ separately can be shown to be accurate to order~$O(n^{-1/2})$ using the
results of Barbour \& Lindvall~(2006).  In a Markov chain, 
the dependence between the states at different times never completely disappears, so we shall need to
make use of the dependence coefficients~$\ch_l$, $1\le l\le 3$.  We make the following simplifying assumption:

\medskip
   {\sl Assumption~A1:}  $\pr[Z_1 = i \giv Z_0 = i] \ >\ 0$ for all $0\le i\le d$.

\medskip
Clearly, a local decomposition in which 
$$
   Z\uj \Def \sum_{|l-j| \le m_n}X\ul \andd Z\ujk \Def \sum_{\max\{|l-j|,|l-k|\}\le m_n}X\ul
$$ 
is likely to be effective, if~$m_n$ is suitably chosen.
Because a finite state irreducible {aperiodic} Markov chain is geometrically ergodic, there 
exist $0 < \r < 1$ and $C < \infty$ such that, for all $0\le i,r\le d$, we have
\eq\label{MC-geometric}
        { |P_{ir}\uk - \p_r| \Le C\r^k, }  
\en
where {$P_{ir}\uk := \pr[Z_k=r \giv Z_0=i]$}, and $\p_r := \lim_{n\to\infty} \pr[Z_n=r \giv Z_0=i]$.
It is then easy to deduce that, as $n\to\infty$,  
\eqa  
     n^{-1}\m_n &\sim& (\p_1,\ldots,\p_d)^T\ =:\ \bpi; \label{MC-variance-asymp} \\
     n^{-1}V_{ir;n} &\sim& \Bigl\{\p_i{\sum_{k\ge1}(P_{ir}\uk - \p_r)} + \p_r{\sum_{k\ge1}(P_{ri}\uk - \p_i)}
                               + \d_{ir} - \p_i\p_r \Bigr\} \ =:\ V_{ir}, \non
\ena
for $1\le i,r \le d$. Now, from~\Ref{MC-geometric}, uniformly in $i,r,s,q$,
\[
    \pr[Z_0=i,Z_j=r,Z_{j+{k}}=s \giv Z_0=i,Z_{j+{k}}=s]
       \Eq \frac{\pr[Z_0=i]P_{ir}\uj {P_{rq}\uk}}{\pr[Z_0=i]P_{iq}^{(j+{k})}} \Eq \p_r\bigl(1 + O(\r^{(j\wedge {k})})\bigr),
\]
and
\eqs
   \lefteqn{\pr[Z_0=i,Z_j=r,Z_{j+l}=s,Z_{j+l+{k}}=q \giv Z_0=i,Z_{j+l+{k}}=q]}\\
     &&\Eq  \frac{\pr[Z_0=i]P_{ir}\uj P_{rs}\ul {P_{sq}\uk}}{\pr[Z_0=i]P_{iq}^{(j+l+{k})}} 
         \Eq \p_r P_{rs}\ul\bigl(1 + O(\r^{(j\wedge {k})})\bigr),
\ens
so that all the dependence coefficients~$\ch_l$, $1\le l\le 3$, in~\Ref{RR-csts} are of order $O(\r^{m_n})$.
It is also immediate, because indicators are bounded random variables, that 
$H_0 = O(1)$, $H_1 = O(m_n)$ and~$H_2 = O(m_n^2)$.  It thus remains to consider the quantity~$\e_{W}$
of~\Ref{dtv-assn}.

Assuming that $2m_n \le n/4$, it is enough to bound
\eq\label{MC-dtv-1}
    \dtv\Bigl(\law_{r}\Bigl(\sum_{j=1}^{l} X\uj + e\uii\Bigr),
                  \law_{r}\Bigl(\sum_{j=1}^{l} X\uj \Bigr)\Bigr),
\en
for any $l \ge \lfloor n/4 \rfloor$ and $0\le r\le d$, where $\law_{r}$ stands for the distribution given the initial state of the Markov chain is at $r$.  This is because, for $1 \le j\le n/2$ and $k \le j+m_n$, 
conditioning on the values of~$Z_j$ up to time $j + k + m_n$ and using the Markov property,
the quantity~$\e_{W}$ in~\Ref{dtv-assn} is no bigger than any bound for the distance in~\Ref{MC-dtv-1},
for $l = n - j - k - m_n \ge n/2 - 2m_n \ge n/4$, that is uniform in the initial state~$r$. 
For $j > n/2$, we note that the same argument
works, using the {\it reversed\/} Markov chain.  We establish~\Ref{MC-dtv-1} by using coupling.

Let $Z'$ and~$Z''$ be two copies of the Markov chain~$Z$, both starting in~$r$.  We couple them in such a way
that the sequence of transitions in the first is the same as that in the second, except that the holding times in $0$
and~$i$ are allowed to be different.  Initially, if $(N_{0l}', \,l\ge1)$ and $(N_{0l}'', \,l\ge1)$ denote the sequence of
successive holding times in~$0$  of the two chains, then the pair $(N_{0l}',N_{0l}'')$ is chosen independently
of the past according to the Mineka coupling (Lindvall 2002, Section~II.14), 
so that $(N_{0l}'-N_{0l}'',\,l\ge1)$ are the increments of a
lazy symmetric random walk with steps in $\{-1,0,1\}$.  After the first occasion~$L_0$ such that 
$$
       \sum_{l=1}^{L_0} \{N_{0l}'-N_{0l}''\} \Eq 1,
$$
the values of $N_{0l}'$ and~$N_{0l}''$ are chosen to be identical.  The same strategy is applied to the holding
times $N_{il}'$ and~$N_{il}''$, except that they are chosen to be identical after the first occasion~$L_i$ on which
$$
     \sum_{l=1}^{L_i} \{N_{il}'-N_{il}''\} \Eq -1.
$$
Let~$M_{0i}$ denote the first time in the underlying Markov chains~$Z'$ and~$Z''$ at which both of these
occasions have occurred.  At this point, both chains have made the same number of steps, because their
paths differ only through differences in the partial sums $\sum_l\{N_{0l}' + N_{il}'\}$ and $\sum_l\{N_{0l}'' + N_{il}''\}$,
and these are equal at all times after~$M_{0i}$.  However, at this point, both have spent the same amount of time in states
other than $i$ and~$0$, but~$Z'$ has spent one step less in~$i$.  By the usual coupling argument, 
for any set $A \subset \Z^d$,
\eqs
  \lefteqn{\pr_r\Bigl[\sum_{j=1}^{{k}} X\uj + e\uii \in A\Bigr] \Eq \pr_r\Bigl[\sum_{j=1}^{{k}} (X\uj)' + e\uii \in A\Bigr]} \\
       &&\Eq \pr_r\Bigl[\Bigl\{\sum_{j=1}^{{k}} (X\uj)' + e\uii \in A\Bigr\} \cap \{M_{0i} \le {k}\}\Bigr]
                 + \pr_r\Bigl[\Bigl\{\sum_{j=1}^{{k}} (X\uj)' + e\uii \in A\Bigr\} \cap \{M_{0i} > {k}\}\Bigr] \\
      && \Eq \pr_r\Bigl[\Bigl\{\sum_{j=1}^{{k}} (X\uj)'' \in A\Bigr\} \cap \{M_{0i} \le {k}\}\Bigr]
             + \pr_r\Bigl[\Bigl\{\sum_{j=1}^{{k}} (X\uj)' + e\uii  \in A\Bigr\} \cap \{M_{0i} > {k}\}\Bigr] \\
      && \Eq  \pr_r\Bigl[\Bigl\{\sum_{j=1}^{{k}} (X\uj)'' \in A\Bigr\} \Bigr] 
              + \pr_r\Bigl[\Bigl\{\sum_{j=1}^{{k}} (X\uj)' + e\uii  \in A\Bigr\} \cap \{M_{0i} > {k}\}\Bigr] \\
      &&\qquad\mbox{}    - \pr_r\Bigl[\Bigl\{\sum_{j=1}^{{k}} (X\uj)'' \in A\Bigr\} \cap \{M_{0i} > {k}\}\Bigr] \\
      && \Eq  \pr_r\Bigl[\Bigl\{\sum_{j=1}^{{k}} X\uj \in A\Bigr\} \Bigr]
          + \pr_r\Bigl[\Bigl\{\sum_{j=1}^{{k}} (X\uj)' + e\uii  \in A\Bigr\} \cap \{M_{0i} > {k}\}\Bigr] \\
      &&\qquad\mbox{}   - \pr_r\Bigl[\Bigl\{\sum_{j=1}^{{k}} (X\uj)'' \in A\Bigr\} \cap \{M_{0i} > {k}\}\Bigr].
\ens
It thus follows that
\[
   \dtv\Bigl(\law_{r}\Bigl(\sum_{j=1}^{{k}} X\uj + e\uii\Bigr),
                  \law_{r}\Bigl(\sum_{j=1}^{{k}} X\uj \Bigr)\Bigr) \Le \pr_r[M_{0i} > {k}].
\]

Now we have $\pr[L_0 > l] = O(l^{-1/2})$ and $\pr[L_i > l] = O(l^{-1/2})$, by Lindvall~(2002,  Section~II.14).
Also, because~$Z$ has finite state space, the times between visits to~$0$ and between visits to~$i$
have means $\g_0$ and~$\g_i$ and finite variances $v_0$ and~$v_i$.  So, if~$\t'_{0l}$ denotes the time at which~$Z'$
completes its $l$-th visit to~$0$, we have
\[
     \{M_{0i} > \quarter n\} \subset \{L_0 > \a n\} \cup \{\t'_{0,\a n} > \quarter n\}
                     \cup \{L_i > \a n\} \cup \{\t'_{i,\a n} > \quarter n\} .   
\]
Hence it follows by Chebyshev's inequality that, if $\a\max\{\g_0,\g_i\} < 1/8$, then
\[
     \pr_r[M_{0i} > \quarter n] \Le \pr_r[L_0 > \a n] + \pr_r[L_i > \a n]
             + \frac{\a nv_0}{(\quarter n - \a\g_0 n)^2} + \frac{\a n v_i}{(\quarter n - \a\g_i n)^2}
     \Eq O(n^{-1/2}),
\]
where this order follows for the first pair of terms as above, and the second pair are
of order $O(n^{-1})$.  This shows that $\e_{W} = O(n^{-1/2})$.

\begin{theorem}\label{MC-thm}
 Let $(Z_j,\,j\ge1)$ be an irreducible, aperiodic Markov chain on a finite state space $\{0,1,\ldots,d\}$,
that satisfies Assumption~A1. Let $W_n := (W_{n1},\ldots,W_{nd})^{{T}}$ represent the number of steps spent
in the states $1,2,\ldots,d$ up to time~$n$.  Then, for any $0\le r\le d$,
\[
   \dtv(\law_r(W_n),\DN_d(n\bpi,nV)) \Eq O(n^{-1/2}\log^3n), 
\]
where $\bpi$ and~$V$ are as given in~\Ref{MC-variance-asymp}.
\end{theorem}

\begin{proof}
 We apply Theorem~\ref{DN-approx}, taking $m_n = \log n/\log(1/\r)$, so that $\ch_1 + \ch_2 + \ch_3
= O(n^{-1})$.  Then $\log n\, H_2(\e_{W} + \hn^{-1/2}) = O(n^{-1/2}\log^3n)$ represents the largest order term
in the error bound.  Finally, it follows from~\Ref{MC-geometric} that {$|\ex W_n - n\bpi| = O(1)$}
and that $|\cov(W_n)_{ir} - nV_{ir}| = O(1)$ for each $i,r$ also, so that, to the stated accuracy,
we can replace the mean and covariance by $n\bpi$ and~$V$ respectively.
\end{proof}

\subsection{Maximal points}\label{max-pts}
\def\bal{{\pmb\alpha}}
\def\bbeta{{\pmb\beta}}
\def\bgamma{{\pmb\gamma}}
\def\boeta{{\pmb\eta}}
\def\mmm{{\hat m}}

Given {a configuration $\Xi$ of points} in~$\re^2$, a point $\bal=(\a_1,\a_2)^T\in \Xi$ is called {\it maximal\/} 
if there are no other points $\bbeta=(\b_1,\b_2)^T\in \Xi$ such that $\b_i\ge \a_i$ for $i=1,2$. 
In this example, we take~$\Xi$ to be a realisation of a Poisson point process with intensity $\lambda$ on the triangle
$$
   \G \Def \{\bal=(\a_1,\a_2)^T:\ 0\le \a_2\le 1-\a_1,\ 0\le \a_1\le 1\}.
$$
Letting 
$$
     A_\bal \Def \{(x_1,x_2)^T:\ \a_1\le x_1\le 1-\a_2,\a_2\le x_2\le 1-x_1\}\setminus\{\bal\},
$$
a point $\bal$ of~$\Xi$ is maximal if $\Xi(A_\bal)=0$. The process of maximal points of~$\Xi$ can thus be written as 
{the random point measure} $\Ups(d\bal):=\bone_{[\Xi(A_\bal)=0]}\Xi(d\bal)$, and has mean measure 
$$
        \ups(d\bal) \Def \ex\Ups(d\bal)=\l e^{-\frac12\l(1-\a_1-\a_2)^2}d\a_1 d\a_2.
$$ 
For $0\le b_1<d_1\le b_2<d_2<\infty$, define the strips
$$
      E_i \Def \bigl\{\bal=(\a_1,\a_2)^T: \{(1-d_i\lambda^{-1/2}-\a_1)\vee 0\} 
               \le\a_2<1-b_i\lambda^{-1/2}-\a_1,\ 0\le\a_1\le1-b_i\l^{-1/2} \bigr\},
$$
{parallel to the hypotenuse of~$\G$ and close to it,}
and define $Y_i=\Ups(E_i)$. Our interest is in the approximate joint distribution of $(Y_1,Y_2)^T$.

\begin{proposition}\label{propmaximal1} Let $\phi(x)=e^{-\frac{x^2}{2}}$ and $\mmm_i=\int_{b_i}^{d_i}\phi(x)dx$,
{and define
\eqs
    \s_{ii} &:=& \mmm_i + 2\mmm_i^2\int_0^{b_i}\frac1{\phi(x)}dx
             +2\int_{b_i}^{d_i}\phi(z)dz\int_{b_i}^z\frac1{\phi(y)}dy\int_{y}^{d_i}\phi(x)\,dx \\
            &&-2\mmm_i(\phi( b_i)-\phi(d_i)),\ i=1,2; \\
    \s_{12} &:=& 2\mmm_2\int_{b_1}^{d_1}\phi(z)dz\int_0^z \frac1{\phi(y)}\,dy \\
              &&- \left\{\mmm_1(\phi(b_2)-\phi(d_2))+\mmm_2(\phi(b_1)-\phi(d_1))\right\}.
\ens
Then, as $\l\to\infty$,}
\bea
    \ex Y_i &=& \ups(E_i) \sim \mmm_i\sqrt{\lambda}; \quad \var(Y_i) \ \sim\ \s_{ii}\sqrt{\l},\  i=1,2;
     \quad \cov(Y_1,Y_2) \ \sim\ \s_{12}\sqrt{\l}. \non
\ben
\end{proposition}

\noindent{\bf Proof:} Since $\ups(d\bal)=\ex\Ups(d\bal)=\l \phi(\sqrt{\l}(1-\a_1-\a_2))d\a_1 d\a_2$, we have
\eqs
    \ups(E_i)&=&\l \int_0^{1-b_i\l^{-1/2}}d\a_1\int_{0\vee (1-\a_1- d_i\l^{-1/2})}^{1-\a_1-b_i\l^{-1/2}}
                   \phi\left(\sqrt{\l}(1-\a_1-\a_2)\right)d\a_2.
\ens
By taking $x=\sqrt{\l}(1-\a_1-\a_2)$ and $y=\a_1$, {we obtain
\[
    \ups(E_i) \Eq \sqrt{\lambda}\int_0^{1-b_i\l^{-1/2}}dy\int_{b_i}^{d_i\wedge (\sqrt{\lambda}(1-y))}\phi(x)\,dx,
\]
from which the first claim follows.} 

\begin{figure}
\centerline{\hskip0cm\parbox{8cm}{
  \begin{center} 
   \includegraphics[trim = 70mm 40mm 45mm 25mm, clip,width=0.58\textwidth]{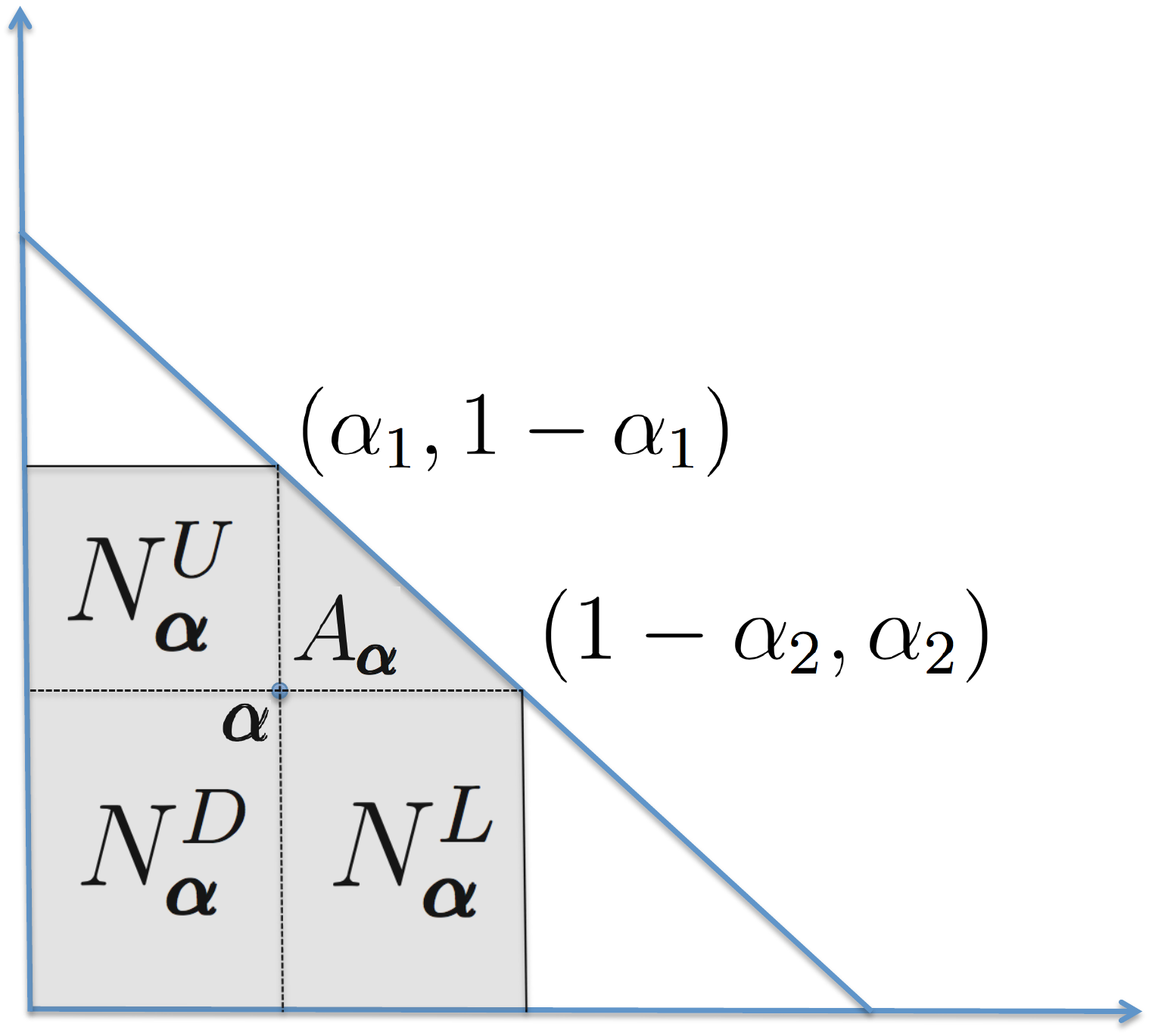}
\end{center}
 \caption{The dependence neighbourhood} %
\label{figure1} 
}
\parbox{8cm}{  \begin{center} 
   \includegraphics[trim = 70mm 40mm 107mm 60mm, clip,width=0.487\textwidth]{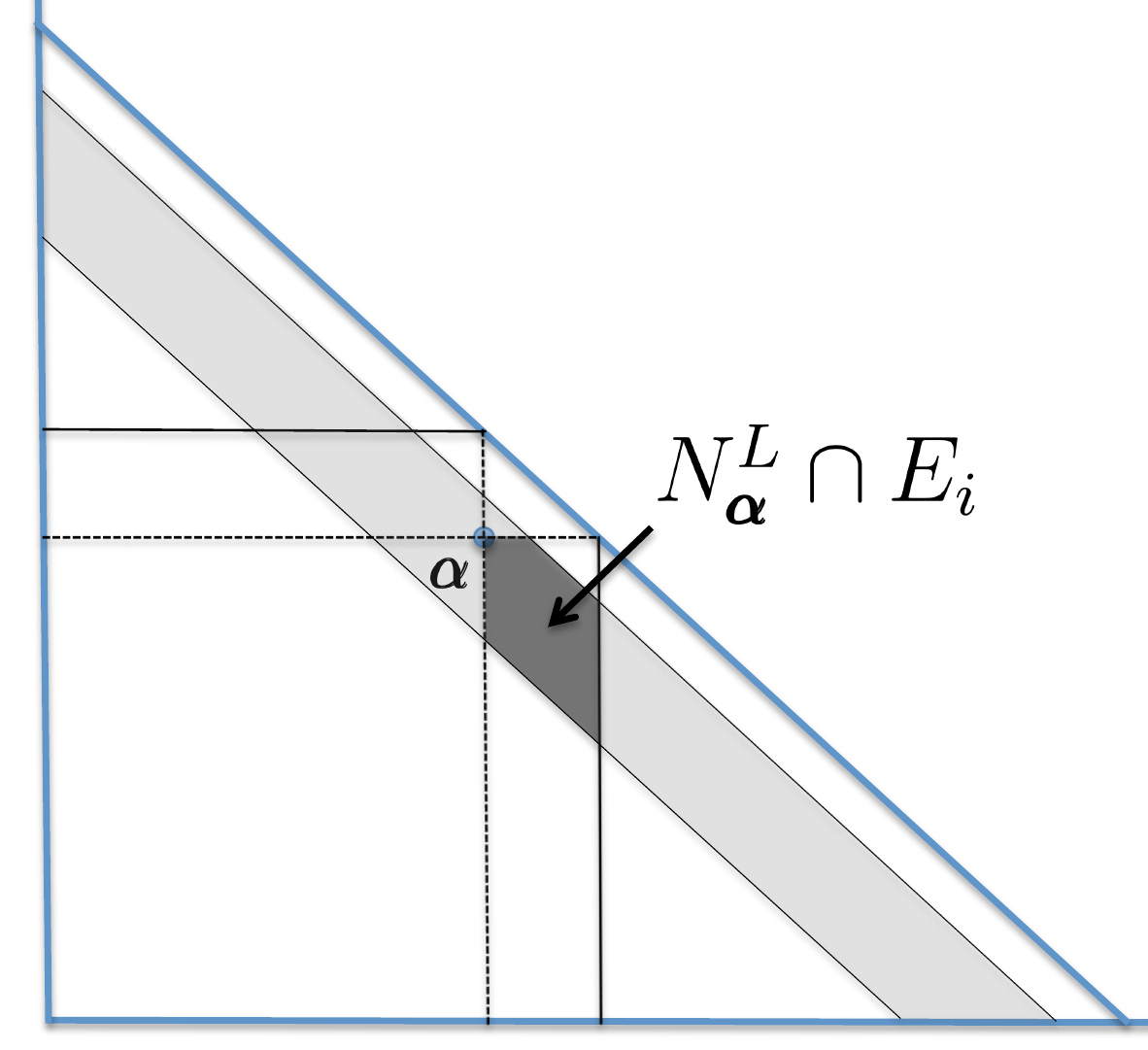}
\end{center}
 \caption{Dark area is $N_\bal^L\cap E_i$} %
\label{figure2} 
}
}
\end{figure}

Next, referring to Figure~\ref{figure1}, we define
\eqs
    N_\bal^U &:=& \{(x_1,x_2)^T:\ 0\le x_1<\a_1,\a_2\le x_2\le 1-\a_1\}; \\
    N_\bal^L &:=& \{(x_1,x_2)^T:\ \a_1\le x_1\le1-\a_2,0\le x_2< \a_2\}; \\
    N_\bal^D &:=& \{(x_1,x_2)^T:\ 0\le x_1<\a_1,0\le x_2<\a_2\}, 
\ens
and then set $N_\bal=A_\bal\cup N_\bal^U\cup N_\bal^L\cup N_\bal^D$.
{Then, since $I[\Xi(A_{\bal}) = 0]$ is independent of $I[\Xi(A_{\bbeta}) = 0]$ for $\bbeta \notin N_{\bal}{\cup\{\bal\}}$,
and $\Xi(N_\bal^D) = \Xi(A_\bal) = 0$ if $\Ups(\{\bal\}) = 1$, we have
\bea
   \var(Y_i)&=& \ups(E_i) + \int_{E_i}\ex\left(\Ups((N_\bal^L \cup N_\bal^U)\cap E_i) 
                     \giv \Ups(\{\bal\})=1\right)\ups(d\bal)\nonumber\\
        &&-\int_{E_i}\ups(N_\bal\cap E_i)\ups(d\bal).\label{maximal04}
\ben
}
However, using Figure~\ref{figure2}, we obtain
\bea
&&\int_{E_i}\ex\left(\Ups(N_\bal^L\cap E_i)|\Ups(\{\bal\})=1\right)\ups(d\bal)\nonumber\\
&& \Eq \lambda\int_{E_i}\ups(d\bal)\int_{\a_1}^{1-\a_2}d\b_1
            \int_{(1-d_i\l^{-1/2}-\b_1)\vee 0}^{\a_2\wedge (1- b_i\l^{-1/2}-\b_1)}
                 e^{-\frac{\lambda}{2}((1-\b_1-\b_2)^2-(1-\b_1-\a_2)^2)}\,d\b_2\nonumber\\
&& \Eq\sqrt{\l}\int_0^{1-b_i\l^{-1/2}}d\a_1\int_{b_i}^{d_i\wedge ((1-\a_1)\sqrt{\l})}\phi(z)\,dz
          \int_0^z\frac1{\phi(y)}dy\int_{y\vee  b_i}^{d_i\wedge (y-z+\sqrt{\l}(1-\a_1))}\phi(x)\,dx\nonumber\\
&&\ \sim\ \sqrt{\l}\mmm_i^2\int_0^{b_i}\frac{1}{\phi(y)}dy+\sqrt{\l}\int_{b_i}^{d_i}\phi(z)\,dz
                \int_{b_i}^{z}\frac1{\phi(y)}dy\int_{y}^{d_i}\phi(x)\,dx,\label{maximal05}
\ben
where the last equality is from the change of variables 
\begin{equation}
   1-\a_2\Eq\a_1+z\lambda^{-1/2},\quad x\Eq(1-\b_1-\b_2)\sqrt{\l}\quad\mbox{and} \quad y\Eq(1-\b_1-\a_2)\sqrt{\l}.
             \label{maximal13}
\end{equation}
{By symmetry, the calculation for $N_\bal^D\cap E_i$ gives an identical result.}
Similarly, by taking $1-\a_2=\a_1+z\lambda^{-1/2}$, $y=\sqrt{\l}(\a_1-\b_1)$ and $x=\sqrt{\l}(1-\b_1-\b_2)$ 
in the second equality below, we get
\bea
    \ups(N_\bal\cap E_i) &=& \int_{(\a_1-d_i\l^{-1/2})\vee 0}^{1-\a_2}d\b_1
      \int_{(1-d_i\l^{-1/2}-\b_1)\vee 0}^{(1-\a_1)\wedge (1- b_i\l^{-1/2}-\b_1)}\l\phi(\sqrt{\l}(1-\b_1-\b_2))d\b_2\,\non\\
    &=& \int_{-z}^{d_i\wedge (\a_1\sqrt{\l})}dy\int_{ b_i\vee y}^{d_i\wedge (y+\sqrt{\l}(1-\a_1))}\phi(x)dx,\non
\ben
which implies that
\bea
   &&\int_{E_i}\ups(N_\bal\cap E_i)\ups(d\bal)\nonumber\\
   &&\Eq \sqrt{\l}\int_0^{1- b_i\l^{-1/2}}d\a_1\int_{b_i}^{d_i\wedge ((1-\a_1)\sqrt{\l})}\phi(z)\,dz
         \int_{-z}^{d_i\wedge (\a_1\sqrt{\l})}dy\int_{ b_i\vee y}^{d_i\wedge (y+\sqrt{\l}(1-\a_1))}\phi(x)\,dx\nonumber\\
   &&\ \sim\ \sqrt{\l}\int_{ b_i}^{d_i}\phi(z)dz\int_{-z}^{d_i}dy\int_{ b_i\vee y}^{d_i}\phi(x)\,dx\nonumber\\
   &&\Eq 2\sqrt{\l}\mmm_i(\phi(b_i)-\phi(d_i)).\label{maximal06}
\ben
Combining \Ref{maximal05} and \Ref{maximal06} with \Ref{maximal04} gives the second claim.

\begin{figure}
\centerline{\hskip0cm\parbox{8cm}{
  \begin{center} 
   \includegraphics[trim = 70mm 40mm 107mm 60mm, clip,width=0.35\textwidth]{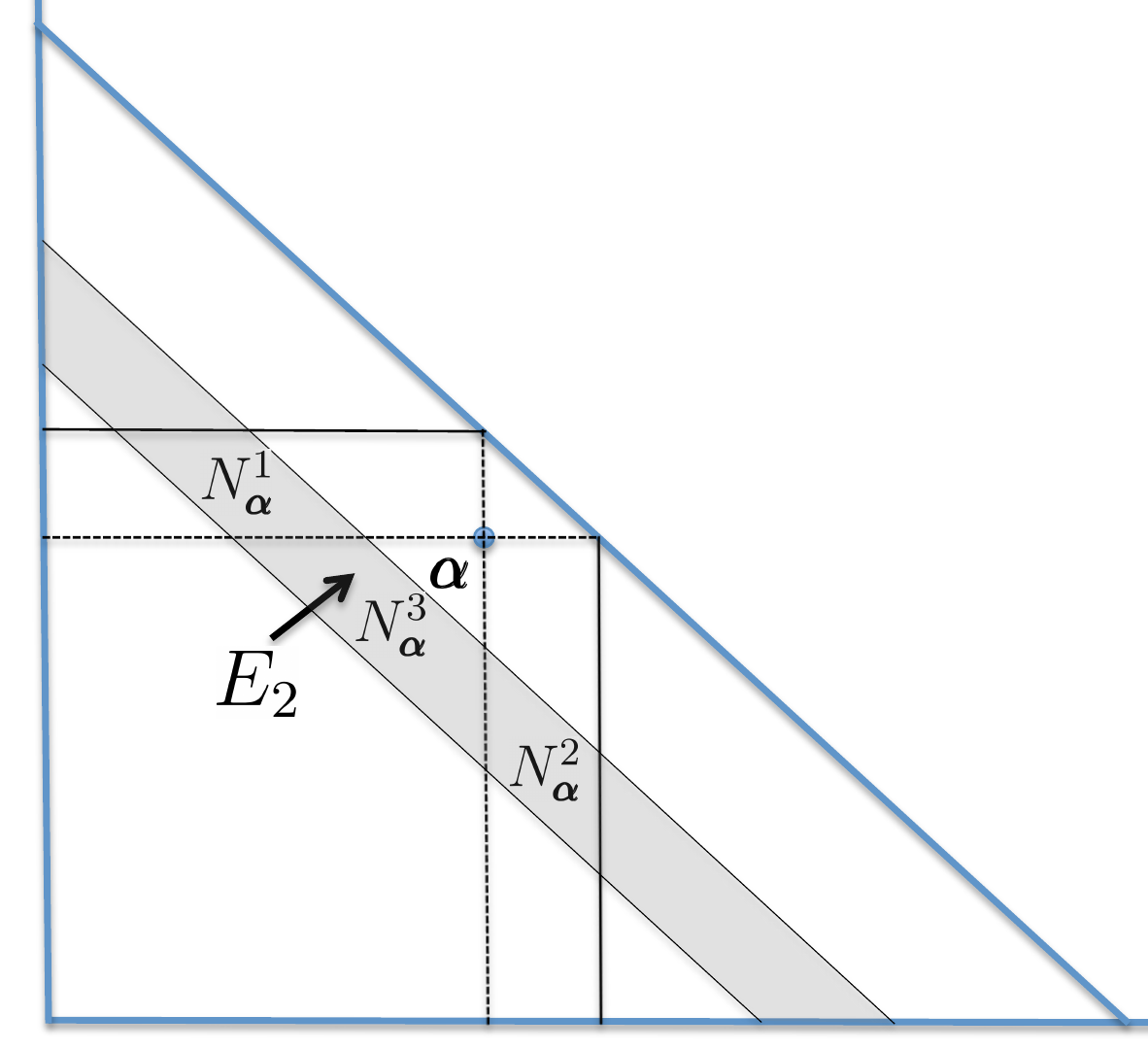}
\end{center}
 \caption{$N_\bal^1,\ N_\bal^2,\ N_\bal^3$}%
\label{figure3} 
}
\parbox{8cm}{  \begin{center} 
  \includegraphics[trim = 2mm 0mm 0mm 0mm, clip,width=0.35\textwidth]{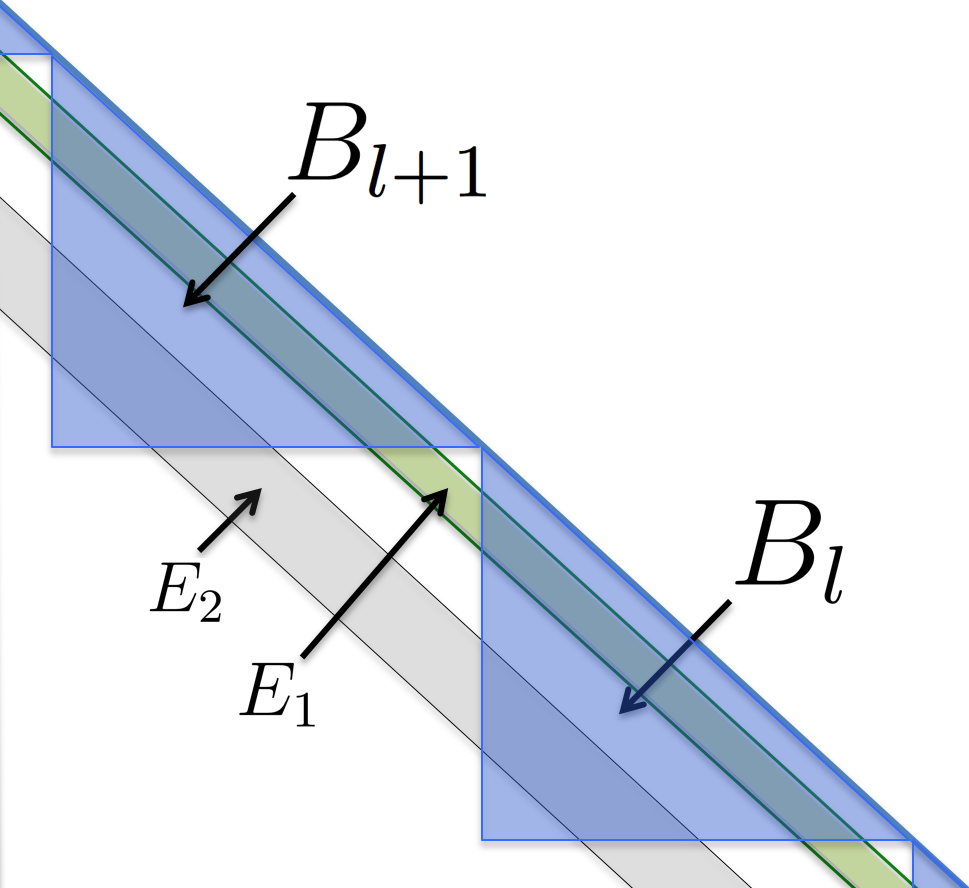}
\end{center}
  \caption{$B_l$} %
\label{figure4} 
}
}
\end{figure}

Finally we estimate $\cov(Y_1,Y_2)$. For $\bal\in E_1$, we refer to Figure~\ref{figure3}  and 
define \break $N_\bal^1:=E_2\cap N_\bal^U$, $N_\bal^2:=E_2\cap N_\bal^L$ and 
$N_\bal^3:=E_2\cap N_\bal^D$. Then we can express the covariance as
\bea
  \cov(Y_1,Y_2)&=& 2\int_{E_1}\ex[\Ups(N_\bal^2)|\Ups(\{\bal\})=1)]\ups(d\bal)
      -\int_{E_1}\ups(N_\bal^1\cup N_\bal^2\cup N_\bal^3)\ups(d\bal).\phantom{XX}\label{maximal07}
\ben
For the first term, we have
\bea
  &&\ex[\Ups(N_\bal^2)|\Ups(\{\bal\})=1)]\nonumber\\
  &&\Eq \l\int_{\a_1}^{1-\a_2}d\b_1\int_{(1-d_2\l^{-1/2}-\b_1)\vee 0}^{1-b_2\l^{-1/2}-\b_1}
                  \frac{\phi(\sqrt{\l}(1-\b_1-\b_2))}{\phi(\sqrt{\l}(1-\b_1-\a_2))}d\b_2\nonumber\\
  &&\Eq \int_0^z\phi^{-1}(y)dy\int_{b_2}^{d_2\wedge (y-z+\sqrt{\l}(1-\a_1))}\phi(x)dx\nonumber\\
  &&\ \sim\ \mmm_2\int_0^z\phi^{-1}(y)dy,\label{maximal08}
\ben
{for $\a_1 < 1$,}
where the last equality is from the change of variables specified in \Ref{maximal13}. It thus follows from
\Ref{maximal08} that
\bea
   &&\int_{E_1}\ex[\Ups(N_\bal^2)|\Ups(\{\bal\})=1)]\ups(d\bal)\nonumber\\
   &&\ \sim\ \mmm_2\sqrt{\l}\int_0^{1-b_1\l^{-1/2}}d\a_1\int_{b_1}^{d_1\wedge ((1-\a_1)\sqrt{\l})}\phi(z)\,dz
              \int_0^z\phi^{-1}(y)\,dy\nonumber\\
   &&\ \sim\ \mmm_2\sqrt{\l}\int_{b_1}^{d_1}\phi(z)\,dz\int_0^z\phi^{-1}(y)\,dy.\label{maximal09}
\ben
Likewise, using the convention that $\int_{c_1}^{c_2}f(x)dx=0$ for $c_1>c_2$, we have
\bea
  \lefteqn{\int_{E_1}\ups(N_\bal^1\cup N_\bal^2\cup N_\bal^3)\ups(d\bal)}\nonumber\\
  &&=\ \l\int_{E_1}\ups(d\bal)\int_{(\a_1-d_2\l^{-1/2})\vee 0}^{1-\a_2}d\b_1
            \int_{(1-d_2\l^{-1/2}-\b_1)\vee 0}^{(1-b_2\l^{-1/2}-\b_1)\wedge (1-\a_1)}\phi(\sqrt{\l}(1-\b_1-\b_2))d\b_2
                  \nonumber\\
  &&=\ \sqrt{\l}\int_0^{1-b_1\l^{-1/2}}d\a_1\int_{b_1}^{d_1\wedge (\sqrt{\l}(1-\a_1))}\phi(z)\,dz
               \int_0^{(z+d_2)\wedge (z+\sqrt{\l}\a_1)}dy
              \int_{b_2\vee (y-z)}^{d_2\wedge (y-z+(1-\a_1)\sqrt{\l})}\phi(x)\,dx\nonumber\\
  &&\sim\ \sqrt{\l}\int_{b_1}^{d_1}\phi(z)dz\int_0^{z+d_2}dy\int_{b_2\vee (y-z)}^{d_2}\phi(x)\,dx\nonumber\\
  &&=\ \sqrt{\l}\mmm_1(\phi(b_2)-\phi(d_2))+\sqrt{\l}\mmm_2(\phi(b_1)-\phi(d_1)),\label{maximal10}
\ben
where, again, we used the the change of variables in \Ref{maximal13} for the penultimate equality.
Combining \Ref{maximal09} and \Ref{maximal10} with~\Ref{maximal07} completes the proof. \qed

\begin{theorem}\label{Maximal-approx}
Let $W=(Y_1,Y_2)^T$, $\m=\ex W$ and $V=\cov(W)$ be as in Proposition~\ref{propmaximal1}.  Then,
as $\l\to\infty$, 
\eqs
   \dtv(\law(W),\DN_2(\m,V)) \Eq O\left(\l^{-1/4}\ln (\l)\right).
\ens
\end{theorem}

\nin{\bf Proof} {In order to apply Theorem~\ref{DN-approx-cont} to the maximal points in $E':=E_1\cup E_2$, we need to
establish suitable decompositions.
As neighbourhoods,} we take $D_\bal:=N_{\bal}\cup\{\bal\}$, with $N_{\bal}$ as defined in 
the proof of Proposition~\ref{propmaximal1} (see Figure~\ref{figure1}). 
Proposition~\ref{propmaximal1} ensures that $\tr\,V {\asymp} \lambda^{1/2}$, {and so $\hn \asymp \lambda^{1/2}$} 
also. We assign a mark 
$$
         X^{(\bal)} \Def \bone_{[\Xi(A_{\bal})=0]}(\bone_{[{\bal}\in E_1]},\bone_{[{\bal}\in E_2]})^T
$$ 
if $\Xi(\{\bal\})=1$, so that $\mu^{(\bal)}=\ex X^{(\bal)}$,  and define
$$
         \tilde X^{(\bal,\bbeta)} \Def X^{(\bbeta)}, \qquad \bbeta\in D_{\bal}.
$$ 
Then 
$$
   W \Eq\int_{\bal\in E'}X^{(\bal)}\Xi(d\bal), \quad\mbox{and}\quad  \nu(d\bal) \Eq \l d\a_1d\a_2,\quad 
                 \nu_2(d\bal,d\bbeta) \Eq \nu(d\bal)\nu(d\bbeta).
$$
We now decompose the integral as follows. For each $\bal\in E'$, define 
\eqs
  Z^{(\bal)} &:=& \int_{\bbeta\in D_{\bal}\cap E'}X^{(\bbeta)}\Xi(d\bbeta);\quad
  W^{(\bal)} \Def \int_{\bgamma\in D_{\bal}^c\cap E'}X^{(\bgamma)}\Xi(d\bgamma);\\
  Z^{(\bal,\bbeta)} &:=& \int_{\bgamma\in D_{\bal}^c\cap D_{\bbeta}\cap E'}X^{(\bgamma)}\Xi(d\bgamma);\quad
  W^{(\bal,\bbeta)} \Def \int_{\bgamma\in D_{\bal}^c\cap D_{\bbeta}^c\cap E'}X^{(\bgamma)}\Xi(d\bgamma).
\ens
This decomposition ensures that $X^{(\bal)}$ is independent of $W^{(\bal)}$ with respect to $\pr$ and $\pr_\bal$, 
and that $W^{(\bal,\bbeta)}$ is independent of $(X^{(\bal)},X^{(\bbeta)})$ 
with respect to $\pr$, $\pr_\bal$, $\pr_\bbeta$ and $\pr_{\bal\bbeta}$.  This immediately implies that
$$
        \ch_{11}' \Eq \ch_{12}' \Eq \ch_{13}' \Eq \ch_{2}' \Eq \ch_3' \Eq \e_W'' \Eq \e_W''' \Eq 0.
$$ 
Hence, it suffices to show that $H_0',\ H_1',\ H_{2i}',\ i=1,2,3,4$, are all of order $O(1)$, 
that~\Ref{Zhat-def-cont} holds and that $\e_{W}'=O(\lambda^{-1/4})$.

For brevity, we write  $\varsigma_{\bal}=\bone_{[\Xi(A_{\bal})=0]}$ and $\theta_{\bal}=\ex (\varsigma_{\bal})$. 
Clearly, $\ex_\bal(\varsigma_{\bal})=\theta_\bal$ also, and 
$$
   \ex_{\bal\bbeta}(\varsigma_{\bal}\varsigma_{\bbeta})\Eq\left\{\begin{array}{ll}
              \ex(\varsigma_{\bal}\varsigma_{\bbeta})&\mbox{ for }\bbeta\in N_\bal^U\cup N_\bal^L\\
                 0&\mbox{ for }\bbeta\in N_\bal^D\cup A_\bal \end{array}\right\}
            \Le \ex(\varsigma_{\bal}\varsigma_{\bbeta}),
$$
where $N_\bal^U,$ $N_\bal^L$ and $N_\bal^D$ are defined in Figure~\ref{figure1}.
Noting that $|X^{(\bal)}|=\varsigma_{\bal}\bone_{[\bal\in E']}$, 
we have
\begin{equation}\ex\int_{\bal\in E'} \varsigma_{\bal}\Xi(d\bal) \Eq \int_{\bal\in E'} \ex_\bal(\varsigma_{\bal})\nu(d\bal)
   \Eq \int_{\bal\in E'} \theta_{\bal}\nu(d\bal) \Eq |\m|_1 \Eq O(\lambda^{1/2}).\label{maximal34}
\end{equation}
It thus follows that
$$
   H_0' \Eq d^{-1/2}\hn^{-1}\int_{\bal\in E'}\ex_\bal (\varsigma_{\bal})\nu(d\bal)
   \Eq d^{-1/2}\hn^{-1}\int_{\bal\in E'}\theta_{\bal}\nu(d\bal) \Eq O(1).
$$
For $H_1'$ and $H_{2i}'$, we repeatedly need to apply the estimate
\begin{equation}
   \nu(D_{\bal}\cap E') \Eq \l O(\lambda^{-1}) \Eq O(1),\label{maximal35}
\end{equation}
{which follows because the area of~$D_{\bal}\cap E'$ is of order $O(\lambda^{-1})$ and the intensity of~$\Xi$
is of order~$O(\l)$.}
Note that the bound \Ref{maximal35} and all the upper bounds below are uniform in $\bal$ and~$\bbeta$.

First, with \Ref{maximal35} in mind, we obtain
\begin{eqnarray*}
    \int_{\bbeta\in D_{\bal}\cap E'}\ex\{\varsigma_{\bbeta}\giv \varsigma_{\bal}=1\}\nu(d\bbeta)
     &\le& \nu(D_{\bal}\cap E') \Eq O(1),\\
    \int_{\bbeta\in D_{\bal}\cap E'}\theta_{\bbeta}\nu(d\bbeta) &\le& \nu(D_{\bal}\cap E') \Eq O(1),
\end{eqnarray*}
which, together with \Ref{maximal34}, imply that
\begin{eqnarray}
  &&\int_{\bal\in E'}\int_{\bbeta\in D_{\bal}\cap E'}
        \{\ex (\varsigma_{\bal}\varsigma_{\bbeta})+\theta_{\bal}\theta_{\bbeta}\}\nu(d\bbeta)\nu(d\bal)\nonumber\\
  &&\Eq \int_{\bal\in E'}\int_{\bbeta\in D_{\bal}\cap E'}
       \left\{\ex \left[\varsigma_{\bbeta}\giv \varsigma_{\bal}=1\right]+\theta_{\bbeta}\right\}
            \theta_{\bal}\nu(d\bbeta)\nu(d\bal)\nonumber\\
  && \Eq O(1)\int_{\bal\in E'}\theta_{\bal}\nu(\bal) \Eq O(\lambda^{1/2}).\label{maximal28}
\end{eqnarray}
It therefore follows from \Ref{maximal28} that
\begin{eqnarray*}
    H_1'&\le&d^{-1}\hn^{-1}\int_{\bal\in E'}\int_{\bbeta\in D_{\bal}\cap E'}
           \{\ex_{\bal\bbeta} (\varsigma_{\bal}\varsigma_{\bbeta})+\theta_{\bal}\ex_\bbeta(\varsigma_{\bbeta})\}
                  \nu(d\bbeta)\nu(d\bal)\\
     &\le& d^{-1}\hn^{-1}\int_{\bal\in E'}\int_{\bbeta\in D_{\bal}\cap E'}
             \left\{\ex(\varsigma_{\bal}\varsigma_{\bbeta})+\theta_{\bal}\theta_{\bbeta}\right\}\nu(d\bbeta)\nu(d\bal) 
                    \Eq O(1).\end{eqnarray*}

To show that $H_{21}'=O(1)$, we proceed as follows. With $\ps_{\bal}:=\Xi( D_{\bal}\cap E')$, we obtain 
from~\Ref{maximal35} that
\begin{equation}
    \ex \left\{|Z^{(\bal)}|^2\right\} \Le \ex\{\ps_{\bal}^2\}
              \Eq \nu(D_{\bal}\cap E')+\nu(D_{\bal}\cap E')^2 \Eq O(1).\label{maximal26}
\end{equation}
Since $\varsigma_{\bal}$ is independent of $\ps'_{\bal}:=\Xi(D_{\bal}\cap A_{\bal}^c\cap E')$, it follows from \Ref{maximal35} and \Ref{maximal26} that
\begin{equation}
   \ex_\bal \{|Z^{(\bal)}|^2\giv\varsigma_{\bal}=1\} 
      \Le \ex_\bal\left\{\left.\left(\ps'_{\bal}\right)^2\right| \varsigma_{\bal}=1\right\}
       \Eq \ex\left\{\left(1+\ps'_{\bal}\right)^2\right\} \Le \ex\{(1+\ps_{\bal})^2\} \Eq O(1).\label{maximal27}
\end{equation}
Combining \Ref{maximal26} and \Ref{maximal27} with~\Ref{maximal34} then ensures that
\begin{eqnarray*}
   H_{21}'&\le&d^{-3/2}\hn^{-1}\int_{\bal\in E'}
           \left\{\ex_\bal (\varsigma_{\bal}|Z^{(\bal)}|^2)+\theta_{\bal}\ex(|Z^{(\bal)}|^2)\right\}\nu(d\bal) \non\\
     &=& d^{-3/2}\hn^{-1}\int_{\bal\in E'}
         \left\{\ex_\bal[|Z^{(\bal)}|^2\giv \varsigma_{\bal}=1]+\ex[|Z^{(\bal)}|^2]\right\}\theta_{\bal}\nu(d\bal) \non\\
     &=& O(\lambda^{-1/2})\int_{\bal\in E'}\theta_{\bal}\nu(d\bal) \Eq O(1).
\end{eqnarray*}
In order to bound $H_{22}'$, $H_{23}'$ and $H_{24}'$, we apply~\Ref{maximal35} again to get the estimates
\begin{eqnarray*}
\ex_\bbeta \{|Z^{(\bal,\bbeta)}|\giv \varsigma_{\bbeta}=1\}
      &\le&\ex_\bbeta\Xi(D_{\bal}^c\cap D_{\bbeta}\cap A_{\bbeta}^c\cap E')
           \Le 1+\nu(D_{\bbeta}\cap E') \Eq O(1),\\
    \ex_{\bal\bbeta} \{|Z^{(\bal,\bbeta)}|\giv\varsigma_{\bal}=\varsigma_{\bbeta}=1\}
      &\le&\ex_{\bal\bbeta}\Xi(D_{\bal}^c\cap D_{\bbeta}\cap A_{\bal}^c\cap A_{\bbeta}^c\cap E')\le2+\nu(D_{\bbeta}\cap E')
           \Eq O(1),\\
     \ex |Z^{(\bal,\bbeta)}|&\le&\ex\Xi(D_{\bal}^c\cap D_{\bbeta}\cap E')
                  \Le \nu(D_{\bbeta}\cap E') \Eq O(1),\\
      \ex |Z^{(\bal)}|&\le&\ex\Xi(D_{\bal}\cap E') \Eq \nu(D_{\bal}\cap E') \Eq O(1).
\end{eqnarray*}
These in turn show that
\eqa
  H_{22}'&\le&d^{-3/2}\hn^{-1}\int_{\bal\in E'}\int_{\bbeta\in D_{\bal}\cap E'}
    \left\{\ex_{\bal\bbeta} (\varsigma_{\bal}\varsigma_{\bbeta}|Z^{(\bal,\bbeta)}|)
          +\theta_{\bal}\ex_\bbeta(\varsigma_{\bbeta}|Z^{(\bal,\bbeta)}|)\right\}\nu(d\bbeta)\nu(d\bal)\nonumber\\
   &=& d^{-3/2}\hn^{-1}\int_{\bal\in E'}
       \int_{\bbeta\in D_{\bal}\cap E'}\ex_{\bal\bbeta}(|Z^{(\bal,\bbeta)}|\giv \varsigma_{\bal}=\varsigma_{\bbeta}=1)
             \pr_{\bal\bbeta}(\varsigma_{\bal}=\varsigma_{\bbeta}=1)\nu(d\bbeta)\nu(d\bal)\nonumber\\
   &&\mbox{} +d^{-3/2}\hn^{-1}\int_{\bal\in E'}\int_{\bbeta\in D_{\bal}\cap E'}
         \ex_\bbeta(|Z^{(\bal,\bbeta)}|\giv\varsigma_{\bbeta}=1)\theta_{\bal}\theta_{\bbeta}\nu(d\bbeta)\nu(d\bal)\nonumber\\
   &=& O(\lambda^{-1/2})\int_{\bal\in E'}\int_{\bbeta\in D_{\bal}\cap E'}
          \{\ex(\varsigma_{\bal}\varsigma_{\bbeta})+\theta_{\bal}\theta_{\bbeta}\} \Eq O(1),\label{maximal29}\\
  H_{23}'&=&d^{-3/2}\hn^{-1}\int_{\bal\in E'}\int_{\bbeta\in D_{\bal}\cap E'}
           \left\{\ex_{\bal\bbeta} (\varsigma_{\bal}\varsigma_{\bbeta})+\theta_{\bal}\ex_\bbeta(\varsigma_{\bbeta})\right\}
             \ex|Z^{(\bal,\bbeta)}|\nu(d\bbeta)\nu(d\bal)  \nonumber\\
   &=& O(\lambda^{-1/2})\int_{\bal\in E'}\int_{\bbeta\in D_{\bal}\cap E'}
      \{\ex(\varsigma_{\bal}\varsigma_{\bbeta})+\theta_{\bal}\theta_{\bbeta}\}\nu(d\bbeta)\nu(d\bal)=O(1),\label{maximal30}
\ena
and
\eqa
  H_{24}'&=&d^{-3/2}\hn^{-1}\int_{\bal\in E'}\int_{\bbeta\in D_{\bal}\cap E'}
          \{\ex_{\bal\bbeta}(\varsigma_{\bal}\varsigma_{\bbeta})+\theta_{\bal}\ex_\bbeta(\varsigma_{\bbeta})\}
               \ex|Z^{(\bal)}|\nu(d\bbeta)\nu(d\bal) \nonumber\\
   &=&O(\lambda^{-1/2})\int_{\bal\in E'}\int_{\bbeta\in D_{\bal}\cap E'}
          \{\ex(\varsigma_{\bal}\varsigma_{\bbeta})+\theta_{\bal}\theta_{\bbeta}\}\nu(d\bbeta)\nu(d\bal)
       \Eq O(1),\label{maximal31}
\ena
where the last equalities in \Ref{maximal29}--\Ref{maximal31} are from \Ref{maximal28}. 

Next, we turn to \Ref{Zhat-def-cont}. In view of \Ref{maximal35} and \Ref{maximal26}, we have
the bounds
\begin{eqnarray*}
   \ex_\bal(|Z^{(\bal)}|^2) &\le& \ex\{(1+\ps_\bal)^2\} \Eq O(1),\quad 
          \ex_\bbeta(|Z^{(\bal)}|^2) \Le \ex\{(1+\ps_\bal)^2\} \Eq O(1),\\
   \ex(|Z^{(\bal,\bbeta)}|^2) &\le& \ex(\ps_\bbeta^2) \Eq O(1), \qquad \qquad
       \ex_\bbeta(|Z^{(\bal,\bbeta)}|^2) \Le \ex\{(1+\ps_\bbeta)^2\} \Eq O(1),\\
  \ex_{\bal\bbeta}(|Z^{(\bal)}|^2) &\le& \ex\{(2+\ps_\bal)^2\} \Eq O(1),\quad 
       \ex_{\bal\bbeta}(|Z^{(\bal,\bbeta)}|^2) \Le \ex\{(1+\ps_\bbeta)^2\} \Eq O(1).
\end{eqnarray*}
To show that both $\ex_\bal(|W-\m|^2)$ and $\ex_{\bal\bbeta}(|W-\m|^2)$ are bounded by 
$\ccc d\hn=O(\l^{1/2})$, for a suitbaly chosen~$\ccc$, we use the following crude estimates,
which are adequate under local dependence conditions:
\eqs
  \ex_\bal(|W-\m|^2) &\le& 2\ex_\bal(|W^{(\bal)}-\m|^2)+2\ex_\bal(|Z^{(\bal)}|^2)
              \Eq 2\ex(|W^{(\bal)}-\m|^2)+2\ex_\bal(|Z^{(\bal)}|^2)\\
     &\le&4\ex(|W-\m|^2) + 4\ex(|Z^{(\bal)}|^2) + 2\ex_\bal(|Z^{(\bal)}|^2) \Eq O(\l^{1/2}),
\ens
and
\eqs
  \ex_{\bal\bbeta}(|W-\m|^2) &\le& 2\ex_{\bal\bbeta}(|W^{(\bal,\bbeta)}-\m|^2)
            + 2\ex_{\bal\bbeta}(|Z^{(\bal)}+Z^{(\bal,\bbeta)}|^2)\\
      &=& 2\ex(|W^{(\bal,\bbeta)}-\m|^2)+2\ex_{\bal\bbeta}(|Z^{(\bal)}+Z^{(\bal,\bbeta)}|^2)\\
      &\le& 4\ex(|W-\m|^2) + 4\ex(|Z^{(\bal)}+Z^{(\bal,\bbeta)}|^2) + 2\ex_{\bal\bbeta}(|Z^{(\bal)}+Z^{(\bal,\bbeta)}|^2)\\
      &=& O(\l^{1/2});
\ens
hence \Ref{Zhat-def-cont} holds.

Finally, we show that $\e_{W}'=O(\lambda^{-1/4})$.  Referring to Figure~\ref{figure4}, we fix $\theta \ge d_2$ as a 
constant, and set $\kappa := \lfloor (\sqrt{\l}/\theta) - 1 \rfloor$.
We then define 
$$
   B_l \Def A_{(1-(l+1)\theta\l^{-1/2},l\theta\l^{-1/2})}; \quad  
   \boeta_l \Def \int_{\bal\in B_l\cap E'}X^{(\bal)}\Xi(d\bal),\quad l=0,1,\dots,\kappa.
$$ 
Then the $\boeta_l$'s are independent and identically distributed random vectors. For any $\bal\in E'$ 
and $\bbeta\in D_{\bal}\cap E'$, there are at most three of $B_l$'s such that 
$B_l\cap(D_{\bal}\cup D_{\bbeta})\ne \emptyset$, so we eliminate such $\boeta_l$'s and define 
$W'_{\bal,\bbeta}:=\sum_{l: B_l\cap(D_{\bal}\cup D_{\bbeta})= \emptyset}\boeta_l$. We  
use $W'_{\bal,\bbeta}$ to estimate $\e_{W}'$. To this end,
let ${\cal F}_{\bal,\bbeta}$ be the $\sigma$-algebra generated by the configurations of points 
of $\Xi$ in $\G\setminus \left(\cup_{l:\ B_l\cap (D_{\bal}\cup D_{\bbeta})=\emptyset} B_l\right)$, 
and let $d_{TV}\left(\left.W^{(\bal)},W^{(\bal)}+e^{(i)}\right|{\cal F}_{\bal,\bbeta}\right)$ denote the total variation
distance between $W^{(\bal)}$ and $W^{(\bal)}+e^{(i)}$ given configurations in ${\cal F}_{\bal,\bbeta}$ 
under $\pr$.  Then it follows that
\eqa
  &&\dtv\bigl(\law(W^{(\bal)} + e\uii \giv X^{(\bal)},Z^{(\bal)}),
                 \law(W\uja \giv X\uja,Z\uja) \bigr)\non\\
   &&\ \ \Le \esup \left\{d_{TV}\left(\left.W^{(\bal)},W^{(\bal)}+e^{(i)}\right|
                {\cal F}_{\bal,\bbeta}\right)\right\}=d_{TV}(W'_{\bal,\bbeta},W'_{\bal,\bbeta}+e^{(i)}),
                \label{maximal32}
\ena
where $\esup$ stands for the essential supremum. Likewise,
\eqa
   &&\dtv\bigl(\law_\a(W^{(\bal)} + e\uii \giv X^{(\bal)},Z^{(\bal)}),
                 \law_\a(W^{(\bal)} \giv X^{(\bal)},Z^{(\bal)}) \bigr) \non\\
                 &&\ \ \ \Le \esup \left\{d_{TV}\left(\left.W^{(\bal)},W^{(\bal)}+e^{(i)}\right|
           {\cal F}_{\bal,\bbeta}\right)\right\}=d_{TV}(W'_{\bal,\bbeta},W'_{\bal,\bbeta}+e^{(i)}),\non\\
  && \dtv\bigl(\law(W^{(\bal,\bbeta)} + e\uii \giv X^{(\bal)},\tX^{(\bal,\bbeta)},Z^{(\bal,\bbeta)}),
                 \law(W^{(\bal,\bbeta)} \giv X^{(\bal)},\tX^{(\bal,\bbeta)},Z^{(\bal,\bbeta)}) \bigr)\non\\
                  &&\ \ \ \Le \esup \left\{d_{TV}\left(\left.W^{(\bal,\bbeta)},W^{(\bal,\bbeta)}+e^{(i)}\right|
                        {\cal F}_{\bal,\bbeta}\right)\right\}=d_{TV}(W'_{\bal,\bbeta},W'_{\bal,\bbeta}+e^{(i)}),\non\\
  && \dtv\bigl(\law_\b(W^{(\bal,\bbeta)} + e\uii \giv X^{(\bal)},\tX^{(\bal,\bbeta)},Z^{(\bal,\bbeta)}),
                 \law_\b(W^{(\bal,\bbeta)} \giv X^{(\bal)},\tX^{(\bal,\bbeta)},Z^{(\bal,\bbeta)}) \bigr)\non\\
                  &&\ \ \ \Le \esup \left\{d_{TV}\left(\left.W^{(\bal,\bbeta)},W^{(\bal,\bbeta)}+e^{(i)}\right| 
                     {\cal F}_{\bal,\bbeta}\right)\right\}=d_{TV}(W'_{\bal,\bbeta},W'_{\bal,\bbeta}+e^{(i)}),\non\\
  && \dtv\bigl(\law_{\a\b}(W^{(\bal,\bbeta)} + e\uii \giv X^{(\bal)},\tX^{(\bal,\bbeta)},Z^{(\bal,\bbeta)}),
                 \law_{\a\b}(W^{(\bal,\bbeta)} \giv X^{(\bal)},\tX^{(\bal,\bbeta)},Z^{(\bal,\bbeta)}) \bigr) \non\\
                  &&\ \ \ \Le \esup \left\{d_{TV}\left(\left.W^{(\bal,\bbeta)},W^{(\bal,\bbeta)}+e^{(i)}\right|
           {\cal F}_{\bal,\bbeta}\right)\right\}=d_{TV}(W'_{\bal,\bbeta},W'_{\bal,\bbeta}+e^{(i)}).\label{maximal33}
\ena

On the other hand, 
$$
   d_{TV}(\boeta_1, \boeta_1+e^{(i)}) \Le 1-\pr(\boeta_1=0)\wedge \pr(\boeta_1=e^{(i)}).
$$ 
Noting that $B_1$ and $B_1\cap E_i$ satisfy
$$
    \nu(B_1) \Eq  \frac{\theta^2}{2},\quad\mbox{and}\quad  \nu(B_1\cap E_i) \Eq \frac1{2}(d_i-b_i)(2\theta-(d_i+b_i)),
$$
we obtain
\eqs
   \pr(\boeta_1=0) &\ge& \pr(\Ups(B_1)=0) \Eq e^{-\frac12\theta^2},\\ 
   \pr(\boeta_1=e^{(i)}) &\ge& \pr(\Ups(B_1\cap E_i)=1,\Ups(B_1\setminus E_i)=0)\\
  &=& \frac12(d_i-b_i)(2\theta-(d_i+b_i))e^{-\frac12\theta^2},
\ens
which together imply that
$$
    d_{TV}(\boeta_1, \boeta_1+e^{(i)}) \Le 1 - \left\{1\wedge\left(\frac12(d_i-b_i)(2\theta-(d_i+b_i))\right)\right\}
         e^{-\frac12\theta^2}.
$$
Hence it follows from Lemma~4.1 of \BLX~(2018b) that
$$
   d_{TV}(W'_{\bal,\bbeta},W'_{\bal,\bbeta}+e^{(i)}) \Eq O\left((\kappa-3)^{-1/2}\right) \Eq O(\l^{-1/4}),
$$
since $\kappa \ge (\sqrt{\l}/\theta) - 2$.
This, together with \Ref{maximal32} and \Ref{maximal33}, ensures that $\e_{W}'=O(\lambda^{-1/4})$, 
and completes the proof of the theorem.  \qed

\section{The proofs of Theorems~\ref{DN-approx} and \ref{DN-approx-cont}}\label{proofmainresults}

Before proving our main theorems, we establish an auxiliary lemma.
It is useful in what follows to be able to extend the definition of a function~$h$ from the ball
$B_{{\hn}\d}(nc)\cap\Z^d$ to the whole of~$\Z^d$ in such a way that $\|\D h\|_\infty$ can be bounded
in terms of~$\|\D h\|_{{3\hn\d/2,\infty}}$.  That this can be done, if ${\hn}\d \ge 2\sqrt d$, is proved 
using the following lemma.

\begin{lemma}\label{h-extension}
 Let~$h \colon \Z^d \to \re$ be given.  Then, for any $x\in \re^d$ and $r>0$, it is possible to 
modify~$h$ outside the set $\Z^d \cap B_{r}(x)$ in such a way that the resulting function~$\tih$
satisfies $\|\D\tih\|_\infty \le \sqrt d \|\D h\|_{{r + \sqrt d,\infty}}$.
\end{lemma}

\begin{proof}
 First,  for all $y = (y_1,\ldots,y_d) \in B_{r}(x)$, we have
\[
  Z(y) \Def \lfloor y \rfloor  + \{0,1\}^d
           \ \subset\ B_{r+\sqrt d}(x),
\]
where $\lfloor y \rfloor := (\lfloor y_1 \rfloor, \ldots, \lfloor y_d \rfloor)$,
because, for each $z \in Z(y)$, $|z-y| \le \sqrt d$.  Extend the definition of~$h$ to all
$y \in B_r(x)$ by averaging over the values at the points~$Z(y)$:
\[
    h(y) \Def \sum_{{\qm} \in \{0,1\}^d} \Bigl\{\prod_{i=1}^d (1 - \{y_i\} 
               + {\qm}_i(2\{y_i\}-1))\Bigr\}\,h(\lfloor y \rfloor + {\qm}) ,
\]
where $\{y_i\} := y_i - \lfloor y_i \rfloor$.  It is immediate that~$h$ is continuous
in $B_r(x)$, and that, for~$y$ in the interior of any unit cube,
\eqs
     |D_j h(y)| &=& \Bigl|\sum_{{\qm} \in \{0,1\}^d} (2{\qm_j}-1)\Bigl\{\prod_{i\ne j} (1 - \{y_i\}
             + {\qm}_i(2\{y_i\}-1))\Bigr\}\,   h(\lfloor y \rfloor + {\qm})\Bigr| \\
    &\le& \sum_{{\qm}' \in \{0,1\}^{j-1}\times\{0\}\times\{0,1\}^{d-j}}\Bigl\{
            \prod_{i\ne j} (1 - \{y_i\} + {\qm}'_i(2\{y_i\}-1))\Bigr\}\,|\D_j h(\lfloor y \rfloor + {\qm}')| \\
    &\le& \|\D h\|_{{r+\sqrt d,\infty}}.
\ens
Hence it follows that $|h(y) - h(y')| \le \sqrt d \|\D h\|_{{r+\sqrt d,\infty}}|y-y'|$ for any $y,y' \in B_r(x)$.

Now define $\tih$ on $\re^d$ by setting $\tih(y) = h(y)$ on~$B_r(x)$, and $\tih(y) = h(\p_x y)$
for $y \notin B_r(x)$, where $\p_x y := x + r(y-x)/|y-x|$ is the projection of~$y$ onto the surface
of~$B_r(x)$.  Then, since 
\[
     |a-b| \ \ge\ \Bigl|\frac{a}{|a|} - \frac{b}{|b|}\Bigr|\quad \mbox{if}\quad |a|,|b| \ge 1,
\]
it follows that 
\[
     |\tih(y) - \tih(y')| \Eq |h(\p_x y) - h(\p_x y')| \Le \sqrt d \|\D h\|_{{r+\sqrt d,\infty}}|\p_x y - \p_x y'|
             \Le \sqrt d \|\D h\|_{{r+\sqrt d,\infty}} |y - y'|,
\]
and so $\|\D \tih\|_\infty \le \sqrt d \|\D h\|_{{r+\sqrt d,\infty}}$. \ep
\end{proof}

We are now in a position to prove our main theorems.

\medskip
\begin{prooftxia} 
We first prove Theorem~\ref{DN-approx-cont}.
Condition~(a) of Theorem~\ref{ADB-DN-approx-thm} follows directly from \Ref{dtv-assn-cont},
with $\e_{W}'$ for~$\e_1$.  We thus turn to Condition~(b), 
using the Stein operator~$\ABA_{\hn}$, as in~\Ref{A-def}, with~$\hn$ as defined in~\Ref{hn-def-cont}.

As a first step, choose some $\d>0$ such that $2\d \le {\d_0}$, where~${\d_0}$ is as in Theorem~\ref{ADB-DN-approx-thm}.
Given any function~$h$ to be used in Theorem~\ref{ADB-DN-approx-thm}(b),
use Lemma~\ref{h-extension} to continue it outside~$B_{3\hn\d/2}(\m)$ in such a way that
\eq\label{delta-extension}
    \|\D h\|_\infty \Le {\sqrt{d}}\|\D h\|_{{2\hn\d,\infty}} \Le {\sqrt{d}}\|\D h\|_{{\hn\d_0,\infty}},
\en 
possible provided that $\sqrt d \le \hn\d/2$;
since the bound given in the theorem is trivial (taking $C_{\ref{DN-approx}} \ge 1$ if necessary) if $\hn \le d^8$,
it is enough for this to suppose that $ \d \sqrt\hn \ge 2$.
We now observe by Cauchy--Schwarz and Chebyshev's inequality that
\eqs
   \lefteqn{|\ex\{(W-\m)^T\D h(W) I[|W-\m| > \hn\d]\}| }\\
     &&\Le \|\D h\|_\infty \{\ex |W - \m|^2 \pr[|W-\m| > \hn\d] \}^{1/2} 
     \Le \|\D h\|_\infty \tr(V)/(\hn\d) \Le d/\d \|\D h\|_\infty.
\ens
This allows the second part of $\ex\{\ABA_{\hn} h(W) I[|W-\m| \le \hn\d]\}$
to be computed without the indicator, at little cost:
\eqa
  \lefteqn{|\ex\{(W-\m)^T\D h(W) I[|W-\m| \le \hn\d]\} - \ex\{(W-\m)^T\D h(W)\}| }\non\\
        &&\Le (d/\d)\,\|\D h\|_\infty. \phantom{XXXXXXXXXXXXXXXXX} \label{Step-1}
\ena
Then, expanding $W$ as a sum and using $\ex_\a X\uja=\m\uja$, we have
\eqa
    \lefteqn{\ex\{(W-\m)^T\D h(W)\}} \\
    &&\qquad\Eq \intG \{\ex_\a\{(X\uja)^T\D h(Z\uja + W\uja)\}-\ex\{(\m\uja)^T\D h(Z\uja + W\uja)\}\}\nu(d\a)\non\\
          &&\qquad\Eq \intG \ex_\a\{(X\uja)^T\,(\D h(Z\uja + W\uja) - \D h(W\uja)) \} \nu(d\a)\non\\
          &&\qquad\ \ \ \ \ \ -\intG \ex\{(\m\uja)^T\,(\D h(Z\uja + W\uja) - \D h(W\uja)) \} \nu(d\a)+ \h_1+\h_2, 
          \label{Step-2}
\ena
where
\eqa
  |\h_1| &=&\left|\intG\ex_\a\left(\left(\ex_\a\left((X\uja)^T\big|W\uja\right)-(\m\uja)^T\right)\D h(W\uja)\right)\nu(d\a)\right|\Le (d{\hn})^{1/2}\|\D h\|_\infty \ch_{11}'\non\\
  |\h_2| &=&\left|\intG(\m\uja)^T\left(\ex_\a \D h(W\uja)-\ex \D h(W\uja)\right)\nu(d\a)\right|\Le 2d^{1/2}\hn\|\D h\|_\infty H_0'\e_W''.\label{Step-1.2}
\ena 

The next step is to approximate $\D h(Z\uja + W\uja) - \D h(W\uja)$ by $\D^2 h(W\uja)Z\uja$
in~\Ref{Step-2}, and to take care of the error.  This is accomplished in a number of steps.
First, in view of Condition~(b) of Theorem~\ref{ADB-DN-approx-thm}, we need {to}
express bounds on the second differences of~$h$ in terms of their supremum
in some ${\hn}\h$-ball around~$\m = {\hn}\cc$;  we do not have an analogue of Lemma~\ref{h-extension} for the
second differences. Thus we re-introduce truncation, to ensure that both $W\uja$ and~$W$
are close enough to~$\m$.  From \Ref{Chi-j-cont} {and~\Ref{Zhat-def-cont},} and by Chebyshev's inequality,
we have
\eqa
   \lefteqn{\ex_\a\bigl\{|(X\uja)^T\{\D h(Z\uja + W\uja) - \D h(W\uja)\}|
      (I[|W\uja - \m| > \hn\dnew] + I[|Z\uja| > \shn])\bigr\}  } \non\\
    && \Le 2\|\D h\|_\infty \bigl\{\ex_\a\{|X\uja| (\pr_\a[|W\uja-\m| >\hn\dnew]+ \hn^{-1}|Z\uja|^2)\}+ \ch_{12\a}'\bigr\}\non\\ 
   && \Le 2\|\D h\|_\infty \bigl\{\ex_\a\{|X\uja| (\pr_\a[|W-\m| > \hn\dnew/2]+ \pr_\a[|Z\uja| > \hn\dnew/2]+ \hn^{-1}|Z\uja|^2)\}+ \ch_{12\a}'\bigr\}\non\\ 
  && \Le 2\|\D h\|_\infty  \bigl\{ \hn^{-1}\ex_\a\{|X\uja| \bigl({8\ccc}\dnew^{-2}d + |Z\uja|^2 \bigr)\} 
                  + \ch_{12\a}' \bigr\}
         \phantom{XXXXXXXXXXX}\label{Step-1.15}
\ena
and
\eqa
   \lefteqn{\ex\bigl\{|(\m\uja)^T\{\D h(Z\uja + W\uja) - \D h(W\uja)\}|
      (I[|W\uja - \m| > \hn\dnew] + I[|Z\uja| > \shn])\bigr\}  } \non\\
    && \Le 2\|\D h\|_\infty \bigl\{\{|\m\uja| (\pr[|W\uja-\m| > \hn\dnew] + \hn^{-1}\ex(|Z\uja|^2))\} \bigr\}\non\\ 
    && \Le 2\|\D h\|_\infty \bigl\{\{|\m\uja| (\pr[|W-\m| > \hn\dnew/2] +\pr[|Z\uja| > \hn\dnew/2]+ \hn^{-1}\ex(|Z\uja|^2))\} \bigr\}\non\\ 
    && \Le 2\|\D h\|_\infty  \hn^{-1}|\m\uja| \bigl({8}\dnew^{-2}d + \ex(|Z\uja|^2) \bigr).
         \phantom{XXXXXXXXXXX}\label{Step-1.15-cont}
\ena

Integrating over~$\a$ with respect to $\nu$, it thus follows from~\Ref{H-defs-cont} that
\eqa
  \lefteqn{\intG \ex_\a \Bigl|(X\uja)^T\{\D h(Z\uja + W\uja) - \D h(W\uja)\} }\non\\
     &&\mbox{} 
       -  (X\uja)^T\{\D h(Z\uja + W\uja) - \D h(W\uja)\}I[|W\uja - \m| \le \hn\dnew] I[|Z\uja| \le \shn]\Bigr|\nu(d\a)
                   \non\\
     &\le& 2d^{3/2}\|\D h\|_\infty ({8\ccc}\dnew^{-2}H_0'+ H_{21}') +  2\|\D h\|_\infty (d\hn)^{1/2}\ch_{12}', \label{Step-3}
\ena
and 
\eqa
  \lefteqn{\intG \ex\Bigl|(\m\uja)^T\{\D h(Z\uja + W\uja) - \D h(W\uja)\} }\non\\
     &&\mbox{} 
       -  (\m\uja)^T\{\D h(Z\uja + W\uja) - \D h(W\uja)\}I[|W\uja - \m| \le \hn\dnew] I[|Z\uja| \le \shn]\Bigr|\nu(d\a)
                   \non\\
     &\le& 2d^{3/2}\|\D h\|_\infty (8\dnew^{-2}H_0' + H_{21}') . \label{Step-3-cont}
\ena
The integrals on the right hand side of~\Ref{Step-2} can thus be replaced by
\eqa
   \lefteqn{\intG \ex_\a\{(X\uja)^T\,(\D h(Z\uja + W\uja) - \D h(W\uja)) \}
             I[|W\uja - \m| \le \hn\dnew] I[|Z\uja| \le \shn] \}\nu(d\a) } \non\\
   \lefteqn{-\intG \ex\{(\m\uja)^T\,(\D h(Z\uja + W\uja) - \D h(W\uja)) \}
             I[|W\uja - \m| \le \hn\dnew] I[|Z\uja| \le \shn] \}\nu(d\a) } \non\\
 &&\Eq \intG \ex_\a\{(X\uja)^T\,\D^2 h(W\uja)Z\uja 
             I[|W\uja - \m| \le \hn\dnew] I[|Z\uja| \le \shn] \}\nu(d\a)\label{Step-2a}\\
&&\mbox{}\qquad + \intG \ex_\a\{(X\uja)^T\,(\D h(Z\uja + W\uja) - \D h(W\uja) - \D^2 h(W\uja)Z\uja)
       \phantom{XXXX}  \non\\
       &&\qquad\qquad \qquad\qquad     I[|W\uja - \m| \le \hn\dnew] I[|Z\uja| \le \shn]\}\nu(d\a) \label{Step-2b}\\
  &&\ \ \ \ \ \ -\intG \ex\{(\m\uja)^T\,\D^2 h(W\uja)Z\uja 
             I[|W\uja - \m| \le \hn\dnew] I[|Z\uja| \le \shn] \}\nu(d\a)\label{Step-2a-cont}\\
&&\mbox{}\qquad -\intG \ex\{(\m\uja)^T\,(\D h(Z\uja + W\uja) - \D h(W\uja) - \D^2 h(W\uja)Z\uja)
       \phantom{XXXX}  \non\\
       &&\qquad\qquad \qquad\qquad     I[|W\uja - \m| \le \hn\dnew] I[|Z\uja| \le \shn]\}\nu(d\a), \label{Step-2b-cont}
\ena
having truncation in both $W\uja$ and~$Z\uja$, with errors bounded by~\Ref{Step-3} and \Ref{Step-3-cont}.  

Now~\Ref{Step-2b} and \Ref{Step-2b-cont} can be represented in terms of sums of second differences of~$h$.
Defining $\hZ^{[\a,l]} := \sum_{t=1}^l Z_t\uja e^{(t)}$, {$1\le l\le d$}, we have
\eqs
   \lefteqn{ (\eii)^T \{\D h(Z\uja + W\uja) - \D h(W\uja) - \D^2h(W\uja)Z\uja\} }\\
          &&\Eq \sum_{l=1}^d\Blb \sum_{s=0}^{Z_l\uja-1}\{\D^2_{il} h(W\uja + \hZ^{[\a,l-1]} + s e\ul) - \D^2_{il}h(W\uja)\}
                     I[Z_l\uja \ge 1] \right.\\
        && \left. \qquad\qquad\mbox{}-\sum_{s=Z_l\uja}^{-1}\{\D^2_{il} h(W\uja + \hZ^{[\a,l-1]} + s e\ul) - \D^2_{il}h(W\uja)\}
                     I[Z_l\uja \le -1] \Brb.
\ens
Writing
\[
    h_{il}(w,\d) \Def \D^2_{il} h(w) I[|w-\m| \le \hn\d],
\]
it then follows that
\eqa
  \lefteqn{ \{\D^2_{il} h(W\uja + \hZ^{[\a,l-1]} + s e\ul) - \D^2_{il}h(W\uja)\} I[|W\uja - \m| \le \hn\dnew] } \non\\
   &=&  h_{il}(W\uja + \hZ^{[\a,l-1]} + s e\ul,\dnew) - h_{il}(W\uja,\dnew) \label{Step-3a}\\
   &&\qquad\mbox{} - \D^2_{il} h(W\uja + \hZ^{[\a,l-1]} + s e\ul) \non\\
    &&\qquad\qquad\mbox{}          \{I[|W\uja + \hZ^{[\a,l-1]} + s e\ul - \m| \le \hn\dnew]
                - I[|W\uja - \m| \le \hn\dnew]\}. \label{Step-3b}
\ena
The contribution from~\Ref{Step-3a} to~\Ref{Step-2b} and \Ref{Step-2b-cont} can be respectively bounded by first taking the expectation 
conditional on $X\uja$ and~$Z\uja$, and using~\Ref{dtv-assn-cont};  this gives
\eqa
   \lefteqn{\bigl| \ex_\a\bigl\{ (X_i\uja)\{h_{il}(W\uja + \hZ^{[\a,l-1]} + s e\ul,\dnew) 
              - h_{il}(W\uja,\dnew) \} 
                  I[|Z\uja| \le \hn\d_2] | \giv X\uja,Z\uja \bigr\} \bigr| } \non\\
    &&\Le |X_i\uja| \,2\|\D^2 h\|_{{\hn\dnew,\infty}}\,\e_{W}'\,(|s| + |\hZ^{[\a,l-1]}|_1),
        \phantom{XXXXXXXXXXXXXXXXXX} \label{Step-4}
\ena
and
\eqa
   \lefteqn{\bigl| \ex\bigl\{ (\ {\m_i\uja})\{h_{il}(W\uja + \hZ^{[\a,l-1]} + s e\ul,\dnew) 
              - h_{il}(W\uja,\dnew) \} 
                  I[|Z\uja| \le \hn\d_2] | \giv X\uja,Z\uja \bigr\} \bigr| } \non\\
    &&\Le {|\m_i\uja|} \,2\|\D^2 h\|_{{\hn\dnew,\infty}}\,\e_{W}'\,(|s| + |\hZ^{[\a,l-1]}|_1).
        \phantom{XXXXXXXXXXXXXXXXXX} \label{Step-4-cont}
\ena
Adding over $s$ and over $1\le i\le d$, integrating over~$\a\in\G$ with respect to~$\nu$ and taking expectations
with respect to $\ex_\a$ and~$\ex$ respectively, we get error bounds of at most
\eq\label{Step-6a}
\e_{W}' \intG  \ex_\a\{|X\uja|_1\,|Z\uja|_1(|Z\uja|_1+1)\} \|\D^2 h\|_{{3\hn\dnew/2,\infty}}\,\nu(d\a)
          \Le 2d^{3} \hn \|\D^2 h\|_{{3\hn\dnew/2,\infty}} H_{21}'\, \e_{W}'
\en
and
\eq\label{Step-6a-cont}
\e_{W}' \intG  \ex\{|\m\uja|_1\,|Z\uja|_1(|Z\uja|_1+1)\} \|\D^2 h\|_{{3\hn\dnew/2,\infty}}\,\nu(d\a)
          \Le 2d^{3} \hn \|\D^2 h\|_{{3\hn\dnew/2,\infty}} H_{21}'\, \e_{W}'.
\en

For the contribution from~\Ref{Step-3b} to~\Ref{Step-2b} and \Ref{Step-2b-cont},  
recalling that $\sqrt\hn \ge 2/\dnew$, we have
\eqs
  \lefteqn{\bigl|I[|W\uja + \hZ^{[\a,l-1]} + s e\ul - \m| \le \hn\dnew]- I[|W\uja - \m| \le \hn\dnew]\bigr|\,
          I[|Z\uja| \le \shn] }\\
  &&\Le I[|W\uja-\m| > \half\hn\dnew] \phantom{XXXXXXXXXXXXXXXXXXXXXXX}
\ens
for $0 \le s < Z_l\uja$ {if $ Z_l\uja\ge 1$, and for $Z_l\uja \le s < 0$ if $ Z_l\uja < 0$.
Arguing for $ Z_l\uja\ge 1$, we thus have
\eqs
   \lefteqn{\Bigl| \sum_{s=0}^{Z_l\uja-1} \D^2_{il} h(W\uja + {\hZ^{[\a,l-1]}} + s e\ul) } \\
   &&\qquad \{I[|W\uja + {\hZ^{[\a,l-1]}} + s e\ul - \m| \le \hn\dnew]- I[|W\uja - \m| \le \hn\dnew]\} 
          I[{|Z\uja|} \le \shn]  \bigr\} \Bigr| \\
   &&\Le {Z\uja_l} I[|W\uja-\m| > \half\hn\dnew] \|\D^2 h\|_{{3\hn\dnew/2,\infty}},
\ens
from which it follows that}
\eqa
    \lefteqn{\Bigl|\sum_{i=1}^d \ex_\a\bigl\{ X_i\uja \sum_{l=1}^d \sum_{s=0}^{Z_l\uja-1} \D^2_{il} 
    h(W\uja + {\hZ^{[\a,l-1]}} + s e\ul) }\non\\
   && \{I[|W\uja + {\hZ^{[\a,l-1]}} + s e\ul - \m| \le \hn\dnew]- I[|W\uja - \m| \le \hn\dnew]\} 
          I[|Z\uja| \le \shn]  \bigr\} \Bigr|  \non\\
    &\le& \|\D^2 h\|_{{3\hn\dnew/2,\infty}} \bigl\{\ex_\a|X\uja|_1\,\pr_\a[|W\uja-\m| > \half\hn\dnew]
                    + \ch_{13\a}'\bigr\}\sqrt{d\hn}
                   \phantom{X} \label{Step-a1} 
\ena
and that
\eqa
    \lefteqn{\Bigl|\sum_{i=1}^d \ex\bigl\{ {\m_i\uja} \sum_{l=1}^d \sum_{s=0}^{Z_l\uja-1} \D^2_{il} 
    h(W\uja + \hZ^{[\a,l-1]} + s e\ul) }\non\\
   && \{I[|W\uja + \hZ^{[\a,l-1]}+ s e\ul - \m| \le \hn\dnew]- I[|W\uja - \m| \le \hn\dnew]\} 
          I[|Z\uja| \le \shn]  \bigr\} \Bigr|  \non\\
    &\le& \|\D^2 h\|_{{3\hn\dnew/2,\infty}} |\m\uja|_1\,\pr[|W\uja-\m| > \half\hn\dnew]
                    \sqrt{d\hn}.
                   \phantom{X} \label{Step-a1-cont} 
\ena
{The argument for $ Z_l\uja < 0$ is almost exactly the same.}

The first part of~\Ref{Step-a1} yields at most
\eqa
  \lefteqn{\sqrt{d\hn} \|\D^2 h\|_{{3\hn\dnew/2,\infty}} d^{1/2} \ex_\a|X\uja| \,\pr_\a[|W\uja-\m| > \half\hn\dnew]
                    }\non\\
    &\le& d\sqrt{\hn} \|\D^2 h\|_{{3\hn\dnew/2,\infty}} \ex_\a|X\uja|\,\{\pr_\a[|W-\m| > \quarter\hn\dnew]
            + \pr_\a[|Z\uja| > \quarter\hn\dnew]\} \non\\
    &\le& 32{\ccc}d^{2}\dnew^{-2}\hn^{-1/2}\|\D^2 h\|_{{3\hn\dnew/2,\infty}} \ex_\a|X\uja| , \label{Step-5}
\ena
and~\Ref{Step-a1-cont} generates at most
\eqa
  \lefteqn{\sqrt{d\hn} \|\D^2 h\|_{{3\hn\dnew/2,\infty}} d^{1/2} |\m\uja| \,\pr[|W\uja-\m| > \half\hn\dnew]
                    }\non\\
    &\le& d\sqrt{\hn} \|\D^2 h\|_{{3\hn\dnew/2,\infty}} |\m\uja|\,\{\pr[|W-\m| > \quarter\hn\dnew]
            + \pr[|Z\uja| > \quarter\hn\dnew]\} \non\\
    &\le& 32d^{2}\dnew^{-2}\hn^{-1/2}\|\D^2 h\|_{{3\hn\dnew/2,\infty}} |\m\uja| , \label{Step-5-cont}
\ena
using Assumption~\Ref{Zhat-def-cont} and Chebyshev's inequality in the last steps.
\ignore{ 
the second part is bounded by 
\eq\label{Step-5a}
          d^{3/2}\hn\|\D^2 h\|_{{3\hn\dnew/2,\infty}}\ch_{13}'.
\en
}
Integrating over~$\a$ with respect to $\nu$, we deduce that the contribution from~\Ref{Step-3b} to~\Ref{Step-2b} 
is bounded by
\eqa
  \lefteqn{ 
   \Bigl({32\ccc}d^2\dnew^{-2} \hn^{-1/2}\intG \ex_\a|X\uja| \nu(d\a)+ d^{3/2}\hn\ch_{13}'\Bigr) \|\D^2 h\|_{{3\hn\dnew/2,\infty}} }\non\\
   &&\Le ( {32\ccc}d^{5/2}{\dnew^{-2}\hn}^{-1/2} H_0' + d^{3/2}\ch_{13}') \,\hn\|\D^2 h\|_{{3\hn\dnew/2,\infty}} 
          \label{Step-6}
\ena
and to~\Ref{Step-2b-cont} is bounded by
\eqa
  \lefteqn{ 
   \Bigl(32d^2\dnew^{-2} \hn^{-1/2}\intG |\m\uja| \nu(d\a)\Bigr) \|\D^2 h\|_{{3\hn\dnew/2,\infty}} }\non\\
   &&\Le ( 32d^{5/2}{\dnew^{-2}\hn^{-1/2}} H_0') \,\hn\|\D^2 h\|_{{3\hn\dnew/2,\infty}} . 
          \label{Step-6-cont}
\ena

This leaves the quantities in~\Ref{Step-2a} and \Ref{Step-2a-cont}.
First, we easily have
\eqs
   \lefteqn{|\ex_\a\{ (X\uja)^T \D^2 h(W\uja) Z\uja
            I[|W\uja - \m| \le \hn\dnew] I[|Z\uja| > \shn]\}|} \\
    && \Le \hn^{-1/2}\|\D^2 h\|_{{\hn\dnew,\infty}} \ex_\a\{|X\uja|_1\,|Z\uja|_1\,|Z\uja|\}, \phantom{XXXXXXXXXXXXXXX}
\ens
and
\eqs
   \lefteqn{|\ex\{ (\m\uja)^T \D^2 h(W\uja) Z\uja
            I[|W\uja - \m| \le \hn\dnew] I[|Z\uja| > \shn]\}|} \\
    && \Le \hn^{-1/2}\|\D^2 h\|_{{\hn\dnew,\infty}} |\m\uja|_1\,\ex\{|Z\uja|_1\,|Z\uja|\}, \phantom{XXXXXXXXXXXXXXX}
\ens  
so that $I[|Z\uja| \le \shn]$ can be dispensed with by incurring an extra error of at most
\eq\label{Step-6.0} 
  2 \hn\|\D^2 h\|_{{\hn\dnew,\infty}}\,d^{5/2}\hn^{-1/2}H_{21}'.  
\en
Then we can expand~$Z\uja$, giving
\eqa
  \lefteqn{ \ex_\a\{ (X\uja)^T \D^2 h(W\uja) Z\uja
            I[|W\uja - \m| \le \hn\dnew] \} }  \non\\
     &&\Eq \int_{D_\a}\ex_{\a\b}\{(X\uja)^T \D^2 h(W\uja) \tX\ujkab 
                    I[|W\uja - \m| \le \hn\dnew]\}\nu_\a(d\b)\phantom{XXX}  \label{Step-6.1}
\ena
and
\eqa
  \lefteqn{ \ex\{ (\m\uja)^T \D^2 h(W\uja) Z\uja
            I[|W\uja - \m| \le \hn\dnew] \} }  \non\\
     &&\Eq \int_{D_\a}\ex_\b\{(\m\uja)^T \D^2 h(W\uja) \tX\ujkab 
                    I[|W\uja - \m| \le \hn\dnew]\}\nu(d\b),\phantom{XXX}  \label{Step-6.1-cont}
\ena
and then introduce the indicator $I[|Z\ujkab| \le \sshn]$ in exchange for an error of at most
\eqa
     &&\|\D^2 h\|_{{\hn\dnew,\infty}}\intG\int_{D_\a}\ex_{\a\b}\{|X\uja|_1\, |\tX\ujkab|_1 |Z\ujkab|/\sshn\} \nu_2(d\a,d\b)\nonumber\\
            &&\qquad \Le 2\dnew^{-1} d^{5/2} H_{22}' \|\D^2 h\|_{{\hn\dnew,\infty}} . \phantom{XX}\label{Step-6.2}\\
     &&\|\D^2 h\|_{{\hn\dnew,\infty}}\intG\int_{D_\a}|\m\uja|_1\ex_\b\{|\tX\ujkab|_1 |Z\ujkab|/\sshn\} \nu(d\b)\nu(d\a)\nonumber\\
            &&\qquad \Le 2\dnew^{-1} d^{5/2} H_{22}' \|\D^2 h\|_{{\hn\dnew,\infty}} . \phantom{XX}\label{Step-6.2-cont}
\ena
The next step is to split $\D^2 h(W\uja)$ in~\Ref{Step-6.1} and \Ref{Step-6.1-cont}, for~$\b\in D_\a$, giving 
\eq\label{Step-1.5}
     \D^2 h(W\uja) \Eq (\D^2 h(W\uja) - \D^2 h(W\ujkab)) + \D^2 h(W\ujkab).
\en
Much as for~\Ref{Step-6}, we write
\eqs
   \lefteqn{ (\D^2 h(W\uja) - \D^2 h(W\ujkab))I[|W\uja - \m| \le \hn\dnew] }\\
   &&\Eq  (\D^2 h(W\uja)I[|W\uja - \m| \le \hn\dnew] - \D^2 h(W\ujkab)I[|W\ujkab - \m| \le \hn\dnew]) \\
    &&\mbox{}\qquad         + \D^2 h(W\ujkab)(I[|W\ujkab - \m| \le \hn\dnew] - I[|W\uja - \m| \le \hn\dnew]).
\ens
Now, using~\Ref{dtv-assn-cont}, we deduce that
\eqa
  \lefteqn{ | \ex_{\a\b}\{(X\uja)^T (\D^2 h(W\uja)I[|W\uja - \m| \le \hn\dnew]
        - \D^2 h(W\ujkab)I[|W\ujkab - \m| \le \hn\dnew]) \tX\ujkab }  \non\\
   &&\qquad  I[|Z\ujkab| \le \sshn]  \giv X\uja,\tX\ujkab,Z\ujkab\}| \phantom{XXXXXXXXXXXXXXXXXXX} \non\\
    &&\Le |X\uja|_1\,2\|\D^2 h\|_{{3\hn\dnew/2,\infty}}\,|\tX\ujkab|_1 |Z\ujkab|_1 \,\e_{W}', \label{Step-7}
\ena
and
\eqa
  \lefteqn{ | \ex_\b\{(\m\uja)^T (\D^2 h(W\uja)I[|W\uja - \m| \le \hn\dnew]
        - \D^2 h(W\ujkab)I[|W\ujkab - \m| \le \hn\dnew]) \tX\ujkab }  \non\\
   &&\qquad  I[|Z\ujkab| \le \sshn]  \giv X\uja,\tX\ujkab,Z\ujkab\}| \phantom{XXXXXXXXXXXXXXXXXXX} \non\\
    &&\Le |\m\uja|_1\,2\|\D^2 h\|_{{3\hn\dnew/2,\infty}}\,|\tX\ujkab|_1 |Z\ujkab|_1 \,\e_{W}', \label{Step-7-cont}
\ena
giving a first contribution to the errors incurred in~\Ref{Step-2a} and \Ref{Step-2a-cont} by splitting $\D^2 h(W\uja)$ in~\Ref{Step-6.1} and \Ref{Step-6.1-cont}
of
\eq\label{Step-6.07}
    2\hn\|\D^2 h\|_{{3\hn\dnew/2,\infty}}\,d^3 H_{22}' \e_{W}'.
\en
For the remaining contribution, because 
\[
   |I[|W\ujkab - \m| \le \hn\dnew] - I[|W\uja - \m| \le \hn\dnew]|\,I[|Z\ujkab| \le \sshn]
         \Le I[|W\ujkab - \m| > \half\hn\dnew],
\]
we have
\eqa
  \lefteqn{  \ex_{\a\b}\{|(X\uja)^T \D^2 h(W\ujkab)(I[|W\ujkab - \m| \le \hn\dnew] - I[|W\uja - \m| \le \hn\dnew])
           \tX\ujkab| 
       I[|Z\ujkab| \le \sshn]  \}  } \non\\
   &\le& \|\D^2 h\|_{{3\hn\dnew/2,\infty}} \bigl(\ex_{\a\b}\{|X\uja|_1\,|\tX\ujkab|_1\}   \pr_{\a\b}[|W\ujkab - \m| > \half\hn\dnew]
           + \ch_{2\a\b}' \bigr). \phantom{XXXXXXXXXX}
                \label{Step-7.5}
\ena
and
\eqa
  \lefteqn{  \ex_\b\{|(\m\uja)^T \D^2 h(W\ujkab)(I[|W\ujkab - \m| \le \hn\dnew] - I[|W\uja - \m| \le \hn\dnew])
           \tX\ujkab| 
       I[|Z\ujkab| \le \sshn]  \}  } \non\\
   &\le& \|\D^2 h\|_{{3\hn\dnew/2,\infty}} \bigl(|\m\uja|_1\,\ex_\b\{|\tX\ujkab|_1\}   \pr_\b[|W\ujkab - \m| > \half\hn\dnew]
           + \ch_{2\a\b}'' \bigr). \phantom{XXXXXXXXXX}
                \label{Step-7.5-cont}
\ena
Integrating over~$\b\in D_\a$ and then~$\a\in\G$, and using Assumption~\Ref{Zhat-def-cont}, the 
first part of~\Ref{Step-7.5} gives at most         
\eqa
   \lefteqn{d\|\D^2 h\|_{{3\hn\dnew/2,\infty}} \intG \int_{D_\a}\ex_{\a\b}\{|X\uja|\,|\tX\ujkab|\}   }
     \non\\
   &&\qquad\mbox{} \{\pr_{\a\b}[|W - \m| > \quarter\hn\dnew] + \pr_{\a\b}[|Z\uja| > \eighth \hn\dnew] 
        + \pr_{\a\b}[|Z\ujkab| > \eighth \hn\dnew]\}\nu_2(d\a,d\b) \phantom{XXXXXXXXX} \non\\
   &\le& \frac{144{\ccc}d^2}{\hn\dnew^{2}} \|\D^2 h\|_{{3\hn\dnew/2,\infty}}\intG \int_{D_\a}  
                            \ex_{\a\b}\{|X\uja|\,|\tX\ujkab|\}\nu_2(d\a,d\b)
       \Eq   \frac{144{\ccc}d^3}{\dnew^{2}} \,H_{1}'  \|\D^2 h\|_{{3\hn\dnew/2,\infty}}     \label{Step-8}
\ena
and the first part of~\Ref{Step-7.5-cont} produces at most         
\eqa
   \lefteqn{d\|\D^2 h\|_{{3\hn\dnew/2,\infty}} \intG \int_{D_\a}|\m\uja|\,\ex_{\b}\{|\tX\ujkab|\}   }
     \non\\
   &&\qquad\mbox{} \{\pr_{\b}[|W - \m| > \quarter\hn\dnew] + \pr_{\b}[|Z\uja| > \eighth \hn\dnew] 
        + \pr_{\b}[|Z\ujkab| > \eighth \hn\dnew]\}\nu(d\b)\nu(d\a) \phantom{XXXXXXXXX} \non\\
   &\le& \frac{144{\ccc}d^2}{\hn\dnew^{2}} \|\D^2 h\|_{{3\hn\dnew/2,\infty}}\intG \int_{D_\a} |\m\uja|\, 
                    \ex_{\b}\{|\tX\ujkab|\}\nu(d\b)\nu(d\a)
       \Eq   \frac{144{\ccc}d^3}{\dnew^{2}} \,H_{1}'  \|\D^2 h\|_{{3\hn\dnew/2,\infty}}.     \label{Step-8-cont}
\ena
The second parts of~\Ref{Step-7.5} and \Ref{Step-7.5-cont} give at most $d^3\hn\|\D^2 h\|_{{3\hn\dnew/2,\infty}}\ch_{2}'$.
Thus \Ref{Step-6.07}, \Ref{Step-7.5}, \Ref{Step-7.5-cont}, \Ref{Step-8} and~\Ref{Step-8-cont} together give a 
contribution to the error of at most
\eq\label{Step-9}
     \hn \|\D^2 h\|_{{3\hn\dnew/2,\infty}}\,\{2d^3 H_{22}'\,\e_{W}' + {288\ccc}d^3\hn^{-1}H_{1}'\dnew^{-2} + d^3\ch_{2}'\}.
\en
Thus, having used~\Ref{Step-1.5} to replace $\D^2 h(W\uja)$ by $\D^2 h(W\ujkab)$ in~\Ref{Step-6.1} and \Ref{Step-6.1-cont},
with the error being bounded by the sum of \Ref{Step-6.2}, \Ref{Step-6.2-cont} and~\Ref{Step-9}, we are left with
\eq\label{Step-9.15}
    \ex_{\a\b}\{(X\uja)^T \D^2 h(W\ujkab) \tX\ujkab I[|W\uja - \m| \le \hn\dnew] I[|Z\ujkab| \le \sshn] \}
\en
and
\eq\label{Step-9.15-cont}
    \ex_{\b}\{(\m\uja)^T \D^2 h(W\ujkab) \tX\ujkab I[|W\uja - \m| \le \hn\dnew] I[|Z\ujkab| \le \sshn] \}.
\en

Exactly as above, we can replace $I[|W\uja - \m| \le \hn\dnew]$ by 
$I[|W\ujkab - \m| \le \hn\dnew]$, adding a second contribution as in~\Ref{Step-8}, \Ref{Step-8-cont} and $d^3\hn\|\D^2 h\|_{{3\hn\dnew/2,\infty}}\ch_{2}'$ to the error.
Then, to remove the factor $I[|Z\ujkab| \le \sshn]$, note that
\eqa\label{Step-9.16}
  \lefteqn{ \ex_{\a\b}\{|(X\uja)^T \D^2 h(W\ujkab) \tX\ujkab| I[|W\ujkab - \m| \le \hn\dnew] I[|Z\ujkab| > \sshn] \} }\non\\
     &&\Le \|\D^2 h\|_{{\hn\dnew,\infty}} \ex_{\a\b}\{|X\uja|_1\, |\tX\ujkab|_1   |Z\ujkab|\} /\sshn,\phantom{XXXXXXX}
\ena
and that
\eqa\label{Step-9.16-cont}
  \lefteqn{ \ex_{\b}\{|(\m\uja)^T \D^2 h(W\ujkab) \tX\ujkab| I[|W\ujkab - \m| \le \hn\dnew] I[|Z\ujkab| > \sshn] \} }\non\\
     &&\Le \|\D^2 h\|_{{\hn\dnew,\infty}} \ex_{\b}\{|\m\uja|_1\, |\tX\ujkab|_1   |Z\ujkab|\} /\sshn.\phantom{XXXXXXX}
\ena
Integrating over~$\b\in D_\a$ and~$\a\in\G$ thus gives a contribution to the error of at most 
\eq\label{Step-9.8}
    2d^{5/2}\dnew^{-1} H_{22}'\|\D^2 h\|_{{\hn\dnew,\infty}}.
\en  
 
After these adjustments, we are left with
\eqa
   \lefteqn{ \intG \int_{D_\a} \ex_{\a\b}\{(X\uja)^T \D^2 h(W\ujkab) \tX\ujkab 
                    I[|W\ujkab - \m| \le \hn\dnew]\}\nu_2(d\a,d\b) }    \label{Step-10}\\
      &&\Eq  \intG \int_{D_\a} \tr\bigl(\ex_{\a\b}((X\uja)(\tX\ujkab)^T)\, 
          \ex_{\a\b}\{\D^2 h(W\ujkab)I[|W\ujkab - \m| \le \hn\dnew]\} \bigr)\nu_2(d\a,d\b) + \h_3  \non  \\  
        &&\Eq  \intG \int_{D_\a} \tr\bigl(\ex_{\a\b}((X\uja)(\tX\ujkab)^T)\, 
          \ex\{\D^2 h(W\ujkab)I[|W\ujkab - \m| \le \hn\dnew]\} \bigr)\nu_2(d\a,d\b) + \h_3 +\h_4 \non                    
\ena
and
\eqa
   \lefteqn{ \intG \int_{D_\a} \ex_{\b}\{(\m\uja)^T \D^2 h(W\ujkab) \tX\ujkab 
                    I[|W\ujkab - \m| \le \hn\dnew]\}\nu(d\b)\nu(d\a) }    \label{Step-10-cont}\\
      &&\Eq  \intG \int_{D_\a} \tr\bigl(\ex_{\b}((\m\uja)(\tX\ujkab)^T)\, 
          \ex_{\b}\{\D^2 h(W\ujkab)I[|W\ujkab - \m| \le \hn\dnew]\} \bigr)\nu(d\b)\nu(d\a) + \h_5 \non \\ 
           &&\Eq  \intG \int_{D_\a} \tr\bigl(\ex_{\b}((\m\uja)(\tX\ujkab)^T)\, 
          \ex\{\D^2 h(W\ujkab)I[|W\ujkab - \m| \le \hn\dnew]\} \bigr)\nu(d\b)\nu(d\a) + \h_5+\h_6.  \non                      
\ena
  One can bound $\h_3$ and~$\h_5$ by
\eq\label{Eta-3-bnd}
   |\h_3|+|\h_5| \Le \hn\|\D^2 h\|_{{\hn\dnew,\infty}}\,d^3\ch_{2}',
\en
and each of $\h_4$ and~$\h_6$ by 
\eq\label{Eta-3'-bnd}
   \max\{|\h_4|,|\h_6|\} \Le 2d^2\hn\|\D^2 h\|_{{\hn\dnew,\infty}} H_{1}'\e_W''.
\en
Since, from~\Ref{dtv-assn-cont}, for any $1\le l,m\le d$, we have
\eqa\label{Step-10x1}
   \lefteqn{|\ex\{\D_{lm}^2 h(W)I[|W - \m| \le \hn\dnew]\} 
            - \ex\{\D_{lm}^2 h(W\ujkab)I[|W\ujkab - \m| \le \hn\dnew]\}| }\non\\
        &&\Le \|\D^2 h\|_{{\hn\dnew,\infty}} (\ex|Z\ujkab|_1 + \ex|Z\uja|_1) \e_{W}', \phantom{XXXXXXXXXXXXXXXXXX}
\ena
we can replace $\ex\{\D^2 h(W\ujkab)I[|W\ujkab - \m| \le \hn\dnew]\}$ by $\ex\{\D^2 h(W)I[|W - \m| \le \hn\dnew]\}$
in~\Ref{Step-10} and \Ref{Step-10-cont}, introducing {further errors} of at most
\eqa
   \lefteqn{\intG \int_{D_\a} \sum_{i=1}^d\sum_{l=1}^d |\ex_{\a\b}\{X_i\uja\tX_l\ujkab\}|
          (\ex|Z\ujkab|_1 + \ex|Z\uja|_1) \nu_2(d\a,d\b)\|\D^2 h\|_{{\hn\dnew,\infty}} \e_{W}' }  \non\\
   &&\Le \|\D^2 h\|_{{\hn\dnew,\infty}} \e_{W}' \intG \int_{D_\a} \ex_{\a\b}\{|X\uja|_1\,|\tX\ujkab|_1\}
           (\ex|Z\ujkab|_1 + \ex|Z\uja|_1) \nu_2(d\a,d\b)\non\\
   &&\Le d^{3}\hn\|\D^2 h\|_{{\hn\dnew,\infty}}(H_{23}' + H_{24}')\e_{W}', \label{Step-10a}
\ena
and
\eqa
   \lefteqn{\intG \int_{D_\a} \sum_{i=1}^d\sum_{l=1}^d |\ex_{\b}\{{\m_i\uja}\tX_l\ujkab\}|
          (\ex|Z\ujkab|_1 + \ex|Z\uja|_1) \nu(d\b)\nu(d\a)\|\D^2 h\|_{{\hn\dnew,\infty}} \e_{W}' }  \non\\
   &&\Le \|\D^2 h\|_{{\hn\dnew,\infty}} \e_{W}' \intG \int_{D_\a} |\m\uja|_1\,\ex_{\b}\{|\tX\ujkab|_1\}
           (\ex|Z\ujkab|_1 + \ex|Z\uja|_1) \nu(d\b)\nu(d\a)\non\\
   &&\Le d^{3}\hn\|\D^2 h\|_{{\hn\dnew,\infty}}(H_{23}' + H_{24}')\e_{W}', \label{Step-10a-cont}
\ena
and leaving the principal term of
\eq\label{Step-10b}
   \ex\{\tr(\hV \D^2 h(W))\,I[|W - \m| \le \hn\dnew]\} ,
\en
where
\eqa
  \hV &\Def& \intG \int_{D_\a} \ex_{\a\b}\bigl(X\uja(\tX\ujkab)^T\bigr) \nu_2(d\a,d\b)- \intG \int_{D_\a} \ex_{\b}\bigl(\m\uja(\tX\ujkab)^T\bigr) \nu(d\b)\nu(d\a)\non\\
        &\Eq& \ex\intG \int_{D_\a} X\uja(\tX\ujkab)^T \Xi(d\b)\Xi(d\a)- \ex\intG \int_{D_\a} \m\uja(\tX\ujkab)^T \Xi(d\b)\nu(d\a)\non\\
         & \Eq& \ex\intG  X\uja(Z\uja)^T\Xi(d\a)-\ex\intG  \m\uja(Z\uja)^T\nu(d\a).
          \label{hV-def}
\ena

We now recall the first term in $\ex\{\ABA_{\hn} h(W) I[|W-\m| \le \hn\d]\}$, which is
\eq\label{Step-11}
    \ex\{\tr(V \D^2 h(W))I[|W-\m| \le \hn\dnew]\} ,
\en 
differing from that in~\Ref{Step-10b} only because the matrix~$V = \cov(W)$ replaces~$\hV$.
If approximation by~$\DN_d(\m,\hV)$ is required,  it is now enough to collect the various errors.
If not, we can write
\eqs
    V \Eq \cov(W) &\Eq& \ex\intG \{X\uja W^T\}\Xi(d\a)-\ex\intG \{\m\uja W^T\}\nu(d\a),
\ens
{so that, recalling $W = W\uja + Z\uja$, we have 
\eqs
    V - \hV &=& \ex\intG \{X\uja (W\uja)^T\}\Xi(d\a)-\ex\intG \{\m\uja (W\uja)^T\}\nu(d\a) \\
       &=& \ex\intG \{X\uja (W\uja-\m)^T\}\Xi(d\a)-\ex\intG \{\m\uja (W\uja-\m)^T\}\nu(d\a). 
\ens
Defining}
\eqs
          V' &:=& \intG \ex_\a\bigl\{(\ex_\a(X\uja \giv W\uja) - \m\uja)(W\uja-\m)^T\bigr\}\nu(d\a),\\
         V'' &:=& \intG \m\uja\bigl\{\ex_\a\left((W\uja)^T\right)-\ex\left((W\uja)^T\right)\bigr\}\nu(d\a),
\ens
we thus have 
\eqs
    V - \hV \Eq V'+V''.
\ens
Hence the difference between \Ref{Step-10b} and~\Ref{Step-11} can be bounded by 
\eqa
    \lefteqn{ \|\D^2 h\|_{{\hn\dnew,\infty}}\sum_{i=1}^d \sum_{l=1}^d (|V'_{il} |+|V''_{il} |)} \non\\
    &&\Le \|\D^2 h\|_{{\hn\dnew,\infty}}\intG \ex_\a\{|\ex_\a(X\uja\giv W\uja)-\m\uja|_1\,|W\uja-\m|_1\} \nu(d\a)\non\\
    &&\ \ \ \ +\|\D^2 h\|_{{\hn\dnew,\infty}}\intG |\m\uja|_1\bigl|\ex_\a\left((W\uja)^T\right)-\ex\left((W\uja)^T\right)\bigr|_1\nu(d\a)\non\\
    && \Le d^2\hn \|\D^2 h\|_{{\hn\dnew,\infty}} \ch_3' 
        +d^{3/2}\hn \|\D^2 h\|_{{\hn\dnew,\infty}} H_0'\e_W''',\label{Step-14-x1}
\ena
{where the second element in~\Ref{Step-14-x1}} is from \Ref{dtvmissed2}.

Adding the error bounds in \Ref{Step-1}, \Ref{Step-1.2}, \Ref{Step-3}, \Ref{Step-3-cont}, \Ref{Step-6a}, \Ref{Step-6a-cont}, \Ref{Step-6}, \Ref{Step-6-cont}
\Ref{Step-6.0}, \Ref{Step-6.2}, \Ref{Step-6.2-cont}, \Ref{Step-9}, \Ref{Step-9.8}, \Ref{Eta-3-bnd}, \Ref{Eta-3'-bnd},
\Ref{Step-10a}, \Ref{Step-10a-cont} and~\Ref{Step-14-x1}, {using \Ref{delta-extension} and with}  {$\shn \ge 2/\d$}, 
gives
\eqs
  \lefteqn{|\ex\{\ABA_{\hn} h(W) I[|W-\m| \le \hn\d]\}| }\\
     &\le& C_1(\d)\{\hn^{-1/2}d^{3/2}(1+ H_0' + H_2') + d^{1/2}\ch_1'+(dm)^{1/2}H_0'\e_{W}''\}\, 
                     \hn^{1/2}\|\D h\|_{{\hn\d_0,\infty}}  \\
      &&\mbox{}\quad     + C_2(\d)\bigl\{\e_{W}' d^3H_2' +{d^3 \hn^{-1}}H_1' 
         + \hn^{-1/2} d^{5/2} (H_0'+H_2')  \\
      &&\mbox{} \qquad \qquad \qquad 
               +d^{3/2}\ch_1' + d^3\ch_2' + d^2\ch_3' +d^2H_1'\e_{W}''+d^{3/2}H_0'\e_{W}'''\bigr\}\, \hn\|\D^2 h\|_{{3\hn\dnew/2,\infty}}. 
\ens
Recalling Theorem~\ref{ADB-DN-approx-thm}, Theorem~\ref{DN-approx-cont} follows. 

{Theorem~\ref{DN-approx} can be deduced from Theorem~\ref{DN-approx-cont} directly by taking 
$\G=\{1,\dots,n\}$, $\Xi$ as the counting measure on $\G$ so that $\Xi(\{i\})=\nu(\{i\})=1$ for all $i\in \G$ and $\nu_2(\{i\},\{j\})=1$ for all $i,j\in\G$; replacing 
$\int$ with $\sum$; $\a$ with $j$, $\b$ with $k$; $\ex_\a$,
$\ex_\b$, $\ex_{\a\b}$ with $\ex$; $\pr_\a$, $\pr_\b$,
$\pr_{\a\b}$ with $\pr$ so that $\e_{W}''=\e_{W}'''=0$; $\ch'$, $H'$, $\e_{W}'$
 with $\ch$, $H$, $\e_{W}$.}
\end{prooftxia}

\end{document}